\documentclass[a4paper,12pt]{article}
\usepackage[breaklinks=true]{hyperref}
\usepackage{amssymb,amsmath,amsthm, amsfonts}
\usepackage{a4wide}
\usepackage{color}

\newenvironment{keywords}{{\bf Key words. }}{\\}
\newenvironment{AMS}{{\bf Mathematics subject classification. }}{\\}

\newcommand{\at}[1]{}

\newcommand{\Gr}{{\rm gph\,}}
\newcommand{\epi}{{\rm epi\,}}

\newcommand{\co}{{\rm conv\,}}

\newcommand{\cl}{{\rm cl\,}}

\newcommand{\xb}{\bar x}
\newcommand{\ub}{\bar u}
\newcommand{\yb}{\bar y}
\newcommand{\zb}{\bar z}
\newcommand{\wb}{\bar w}

\newcommand{\yba}{\yb^\ast{}}
\newcommand{\vb}{\bar v}
\newcommand{\vba}{\vb^\ast{}}
\newcommand{\zba}{\zb^\ast{}}

\newcommand{\lb}{{\bar\lambda}}

\newcommand{\Q}{{\cal Q}}

\newcommand{\Lsp}{{\cal L}}

\newcommand{\KbG}{{\bar K_\Gamma}}
\newcommand{\Tlin}{T^{\rm lin}}
\newcommand{\TlinO}{\Tlin_{P,D}}
\newcommand{\TlinOk}[1]{T^{{\rm lin},#1}_{P,D}}

\newcommand{\Lb}{\bar\Lambda}
\newcommand{\Lbv}{\bar\Lambda(v)}

\newcommand{\R}{\mathbb{R}}

\newcommand{\norm}[1]{\|#1\|}

\newcommand{\dist}[1]{{\rm d}(#1)}
\newcommand{\B}{{\cal B}}

\newcommand{\K}{{\cal K}}

\newcommand{\mv}{\,\vert\, }

\newcommand{\longsetto}[1]{\mathop{\longrightarrow}\limits^#1}
\newcommand{\oo}{o}

\newcommand{\skalp}[1]{\langle #1\rangle}

\newcommand{\argmax}{\mathop{\rm arg\,max}\limits}

\newtheorem{definition}{Definition}
\newtheorem{theorem}{Theorem}
\newtheorem{lemma}{Lemma}
\newtheorem{example}{Example}
\newtheorem{corollary}{Corollary}
\newtheorem{proposition}{Proposition}
\newtheorem{remark}{Remark}
\newtheorem{assumption}{Assumption}
\begin{document}
\title{Linearized M-stationarity conditions for general optimization problems}
\author{Helmut Gfrerer\thanks{Institute of Computational Mathematics, Johannes Kepler University (JKU) Linz, A-4040 Linz, Austria, e-mail:
helmut.gfrerer@jku.at.}}
\date{}

\maketitle
\begin{abstract}
  This paper investigates new first-order optimality conditions for general optimization problems. These optimality conditions are stronger than the commonly used M-stationarity conditions and are in particular useful when the latter cannot be applied because the underlying limiting normal cone cannot be computed effectively. We apply our optimality conditions to a MPEC to demonstrate their practicability.
\end{abstract}



\section{Introduction}
This paper deals with first-order optimality conditions for general optimization problems of the form
\begin{equation}\label{EqGenOptProbl}
  \min_z f(z)\quad\mbox{subject to}\quad P(z)\in D
\end{equation}
where the mappings $f:\R^d\to \R$ and $P:\R^d\to\R^s$ are assumed to be continuously differentiable and $D$ is a closed subset of $\R^s$.  Note that formally more general problems of the form
\begin{eqnarray}\label{EqGenOptProbl1}\min_{z}&&f(z)\\
\nonumber\mbox{subject to}&&0\in P(z)+Q(z),
\end{eqnarray}
where $Q:\R^d\rightrightarrows\R^s$ is a set-valued mapping with closed graph, can be equivalently written in the form \eqref{EqGenOptProbl} as
\begin{equation}\label{EqGenOptProbl2}\min f(z)\quad\mbox{subject to}\quad (z,-P(z))\in\Gr Q.\end{equation}
If the objective function in \eqref{EqGenOptProbl} is not continuously differentiable, we can equivalently rewrite the program \eqref{EqGenOptProbl} as
\begin{equation}\label{EqGenOptProbl3}\min_{z,\alpha} \alpha\quad\mbox{subject to}\quad (z,\alpha,P(z))\in\epi f\times D.\end{equation}
Under some constraint qualification, necessary optimality conditions for the problem \eqref{EqGenOptProbl} at a local minimizer $\zb$ are usually of the form
\begin{equation}\label{EqKKT1}0\in \nabla f(\zb)+\nabla P(\zb)^\ast w^\ast,\end{equation}
where the multiplier $w^\ast$ belongs to a suitable normal cone  to the set $D$ at the point $P(\zb)$, which in turn is often related to the notion of a subdifferential. Among the big number of different normal cones/subdifferential constructions considered in the literature, two stand out by the comprehensive calculus available for them: One is given by the {\em generalized gradient} as introduced by Clarke \cite{Cla73} and the related normal cone, the other one is the {\em limiting  (Mordukhovich) normal cone/subdifferential}. Since the Clarke normal cone is the closure of the convex hull of the limiting normal cone, c.f. \cite{RoWe98}, the use of the limiting normal cone yields stronger first-order optimality conditions than an approach based on Clarke's normal cone and for this reason we focus in this paper on first-order optimality conditions related to the limiting normal cone, which are usually called M-stationarity conditions. However, despite the available calculus, it is sometimes very difficult or even impossible to compute the limiting normal cone effectively.

As an illustrating example let us consider the following subclass of so-called {\em mathematical programs with equilibrium constraints} (MPECs), where the equilibrium is  described by a generalized equation:
\begin{align}
\label{EqMPEC}\mbox{(MPEC)}\qquad \min_{x,y}\ & F(x,y)\\
\nonumber  \mbox{s.t. }&0\in\phi(x,y)+\widehat N_\Gamma(y),\\
\nonumber  &G(x,y)\leq 0
\end{align}
For this problem, the mappings $F:\R^n\times\R^m\to \R$, $\phi:\R^n\times\R^m\to \R^m$ and $G:\R^n\times\R^m\to\R^p$ are assumed to be continuously differentiable, $\Gamma:=\{y\mv g(y)\leq 0\}$ is given by a $C^2$-mapping $g:\R^m\to\R^q$ and $\widehat N_\Gamma(y)$ denotes the {\em regular (Fr\'echet) normal cone} to $\Gamma$ at $y$, cf. Definition \ref{DefCones} below. The program (MPEC) can be equivalently written in the format \eqref{EqGenOptProbl} as
\begin{align}
\label{EqMPEC'}\mbox{(MPEC')}\qquad \min_{x,y}\ & F(x,y)\\
\nonumber  \mbox{s.t. }&\hat P(x,y):=\left(\begin{array}{c}(y,-\phi(x,y))\\ G(x,y)\end{array}\right)\in\Gr\widehat N_\Gamma\times\R^p_-=:\hat D
\end{align}
The calculation of the limiting normal cone to $\hat D$ at $\hat P(\xb,\yb)$ involves the one of the limiting normal cone to $\Gr\widehat N_\Gamma$ at $(\yb,-\phi(\xb,\yb)$. The latter task is well-understood, if for the inequalities $g(y)\leq 0$ the {\em linear independence constraint qualification (LICQ)} is fulfilled at $\yb$, cf. \cite{MoOut01}. The situation, unfortunately, becomes substantially more difficult, provided LICQ is relaxed. Such a situation has been investigated under Mangasarian-Fromovitz constraint qualification (MFCQ) in \cite{HenOutSur09} and, under a certain constraint qualification less restrictive than MFCQ, in \cite{GfrOut16a}.  In both cases an additional condition is needed to obtain a point based representation of the limiting normal cone to $\Gr\widehat N_\Gamma$ in terms of first-order and second-order derivatives of $g$ at $\yb$ and in \cite{GfrOut16a} a simple example is given that without this additional condition the limited normal cone cannot be entirely expressed in terms of first-order and second-order derivatives of $g$.

On the other hand, very recently much progress has been achieved in computing the tangent cone to $\Gr\widehat N_\Gamma$ and to the tangent cone of the feasible region of \eqref{EqMPEC}, see \cite{GfrOut16b, ChiHi17, GfrYe17a}. Under very mild assumptions one obtains a full description of the tangent cone to the feasible region of \eqref{EqMPEC} involving only first-order derivatives of $\phi$, $G$ and derivatives of $g$ up to second-order at a point $(\xb,\yb)$. Thus there must exist also some dual optimality condition in terms of these derivatives showing that the part of the limiting normal cone which is difficult to compute does not play a role in the optimality conditions.

At this point let us mention that it might be not feasible to reformulate the MPEC \eqref{EqMPEC} as a {\em mathematical program with complementarity constraints (MPCC)},
\begin{align}
\label{EqMPCC}\qquad \min_{x,y,\lambda}\ & F(x,y)\\
\nonumber  \mbox{s.t. }&0\in\phi(x,y)+\nabla g(y)^\ast\lambda,\\
\nonumber &0\leq \lambda \perp g(y)\geq 0,\\
\nonumber  &G(x,y)\leq 0.
\end{align}
Of course, if $(\xb,\yb)$ is a local solution of \eqref{EqMPEC} and the system $g(y)\leq 0$ fulfills some constraint qualification at $\yb$ ensuring $\widehat N_\Gamma(\yb)=\{\nabla g(\yb)^\ast\lambda\mv 0\leq \lambda\perp g(\yb)\}$, then it is easy to show that for every multiplier $\lb\geq 0$ fulfilling $0\in \phi(\xb,\yb)+\nabla g(\yb)^\ast\lb$, $\lb^Tg(\yb)=0$ the triple
$(\xb,\yb,\lb)$ is a local solution of \eqref{EqMPCC}. However, if LICQ fails to hold for the system $g(y)\leq 0$ at $\yb$, then it can happen that some constraint qualification is fulfilled for the MPEC \eqref{EqMPEC}, but all of the  MPCC-tailored constraint qualifications known from the literature are violated for \eqref{EqMPCC}. Thus we cannot apply the known first-order optimality conditions for the program \eqref{EqMPCC} in order to obtain optimality conditions for the program \eqref{EqMPEC}. This was first observed in \cite{AdHenOut17} and further developed in \cite{GfrYe17a}. In the latter paper an example is given where this phenomena occurs for convex quadratic functions $g_i$, $i=1,\ldots, q$ and linear mappings $\phi$ and $G$.

To overcome the difficulties arising when computing the limiting normal cone, we remember that the basic task in formulating first-order optimality conditions is the computation of the regular normal cone to the feasible set of \eqref{EqGenOptProbl}. However, for the regular normal cone only very restricted calculus is available and this is the reason why the limiting normal cone is used instead of the regular one. Having in mind that the basic goal is the computation of the regular normal cone to the feasible set, it is not difficult to see that in order to obtain a more accurate approximation we can use  the limiting normal cone to the tangent cone of the feasible set. Performing a more accurate analysis we observe that this process can be repeated and we obtain as a final result that the multiplier $w^\ast$ in \eqref{EqKKT1}  is a regular normal to a series of tangent cones to tangent cones to the set $D$. Since the new optimality conditions are derived by a repeated linearization procedure, we call the resulting optimality conditions {\em linearized M-stationarity conditions}.

The organization of the paper is as follows. In Section \ref{SecVarAna} we recall some basics from variational analysis. The  stationarity concepts of
B-,S- and M-stationarity and its relations with necessary optimality conditions are considered in Section \ref{SecStat}.

Section \ref{SecLM_stat} contains the main results on linearized M-stationarity conditions for the problem \eqref{EqGenOptProbl}. The analysis is done under a very weak constraint qualification: We only require the {\em generalized Guignard constraint qualification (GGCQ)} and the {\em metric subregularity constraint qualification (MSCQ)} for the linearized problem. In particular, both conditions are fulfilled if MSCQ holds for the problem \eqref{EqGenOptProbl}.

We apply these results to the MPEC \eqref{EqMPEC} in Section \ref{SecMPEC} and derive the linearized M-stationarity conditions under a certain  condition on the lower level system $q_i(y)\leq 0$, $i=1,\ldots,p$, which is weaker than the {\em constant rank constraint qualification} (CRCQ). This also works when we are not able to compute the limiting normal cone to $\Gr \widehat N_\Gamma$ as in \cite{GfrOut16a}.

In the concluding Section \ref{SecConcl} we briefly summarize the obtained results  and outline some topics for our future research.

Throughout the paper we use standard notation of variational analysis and generalized
differentiation. For an element $z\in\R^d$ we denote by $[z]$ the subspace $\{\alpha z\mv\alpha\in\R\}$ generated by $z$. Some more special symbols are introduced when appearing first in the text.

\if{
\begin{subequations}
\label{SubEqMPEC}
\begin{align}\label{SubEqMPEC_obj}\min_z& f(z)\\
\label{SubEqMPEC_equil}\text{subject to }&(G(z),H(z))\in \Gr \widehat N_\Gamma,\\
\label{SubEqMPEC_inequ}&g(z)\in Q,
\end{align}
\end{subequations}
where $g:\R^d\to \R^{m_Q}$ and $G,H:\R^d\to \R^{m_\Gamma}$ are continuously differentiable, $Q\subseteq \R^{m_Q}$ is a closed convex set, $\Gamma\subseteq \R^{m_\Gamma}$ is a closed set and $\widehat N_\Gamma$ denotes the regular normal cone mapping to $\Gamma$, whose definition can be found in the next section.

The MPEC \eqref{SubEqMPEC} can be obviously written in the format \eqref{EqGenOptProbl} by putting
\[P(z):=(G(z),H(z),g(z)),\quad D:=\Gr \widehat N_\Gamma\times Q.\]
In case that $\Gamma$ is a closed convex cone $K$, the MPEC \eqref{SubEqMPEC} can be rewritten as a {\em mathematical program with complementarity constraints} (MPCC)
\begin{subequations}
\label{SubEqMPCC}
\begin{align}\label{SubEqMPCC_obj}\min_z\ & f(z)\\
\label{SubEqMPCC_equil}\text{subject to }&G(z)\in K,\ H(z)\in K^\circ,\ \skalp{G(z),H(z)}=0,\\
\label{SubEqMPCC_inequ}&g(z)\in Q,
\end{align}
\end{subequations}
which corresponds to the setting
\[P(z)=(G(z),H(z),\skalp{G(z),H(z)},g(z)),\ D:=K\times K^\circ\times \{0\}\times Q.\]
The MPCC \eqref{SubEqMPCC} also comprises  the standard MPCC with $K=\R^m_-$ and $Q=\R^{m_Q}_-$. In this special case the MPCC \eqref{SubEqMPCC} is a smooth nonlinear programs which, however, do not satisfy most of the standard constraint qualifications. This has led to several weakened stationarity notions that have
been introduced in connection with optimality conditions and numerical approaches.
Two of these stationary notions play an important role: One is called {\em strong stationarity} (S-stationarity), which guarantees that at the point under consideration no feasible descent direction exists. This is certainly a  first-order necessary optimality condition, however the constraint qualifications needed to ensure S-stationarity of a local minimizer  appear to be too restrictive for MPECs. Another, weaker stationarity notion is referred to by the identifier {\em M-stationarity}. M-stationarity has the advantage that it requires only very weak constraint qualifications, but it does not preclude the existence of feasible descent directions. S-stationarity was first considered in the  monograph by Luo, Pang and Ralph \cite{LuPaRa96} whereas M-stationarity conditions appeared first in the papers by Outrata \cite{Out99} and Ye \cite{Ye99}, respectively. The monikers M-stationarity and S-stationarity were coined in \cite{Sch00,SchSch00}.

In this paper we consider the generalization of the notions of S-stationarity and M-stationarity, which were introduced for standard MPCCs, to the more general problem \eqref{EqGenOptProbl} as presented  by Flegel, Kanzow and Outrata \cite{FleKanOut07}. This approach is based on modern concepts of variational analysis and generalized differentiation and gives deep insight into the relations between the required constraint qualifications and  stationarity conditions. Moreover, this approach gives us also some idea how to strengthen the M-stationarity conditions in order to exclude some spurious feasible descent directions. A first step in this direction has been done in the very recent papers \cite{BeGfr17a,BeGfr16d}, where a new stationarity concept named $\Q$-stationarity was introduced which could be combined with M-stationarity to obtain the stronger notion of $\Q_M$-stationarity. $\Q_M$-stationarity seems to be very useful in case when the set $D$ in \eqref{EqGenOptProbl} has some polyhedral structure. In particular, in \cite{BeGfr16d} also an algorithm is given how to resolve the inherent combinatorial structure of the M-stationarity conditions by means of $\Q$-stationarity.
In this paper we consider a different approach. We consider the case when $D$ is non-polyhedral and we qualify the multipliers which fulfill the M-stationarity conditions at a local minimizer. We will show that not every multiplier belonging to the limiting normal cone of $D$ is suitable, but merely regular normals to a series of tangent cones to  $D$ associated with some critical directions. Moreover, this has the advantage that we need not to calculate the full limiting normal cone but only a part of it which is easier manageable. As an example we consider the case of an MPEC \eqref{SubEqMPEC} where $\Gamma$ is given by $C^2$-inequalities $q_i(y)\leq 0$, $i=1,\ldots,p$ and $Q=\R^{m_Q}_-$ is the negative orthant. Under a certain  condition on the lower level system $q_i(y)\leq 0$, $i=1,\ldots,p$, which is weaker than the {\em constant rank constraint qualification} (CRCQ), we obtain new first-order optimality conditions. This also works when we are not able to compute the limiting normal cone to $\Gr \widehat N_\Gamma$ as in \cite{GfrOut16a}.

The organization of the paper is as follows.}\fi
%
\section{Preliminaries from variational analysis\label{SecVarAna}}
All the sets under consideration are supposed to be locally closed around the points in question without further mentioning. We recall first the standard constructions of variational analysis used in what follows.
\begin{definition}\label{DefCones}
Given a set
$\Omega\subseteq\mathbb R^d$ and a point $\bar z\in\Omega$,
the (Bouligand-Severi) {\em tangent/contingent cone} to $\Omega$
at $\bar z$ is a closed cone defined by
\begin{equation*}\label{normalcone}
T_\Omega(\bar z)
:=\Big\{w\in\mathbb R^d\Big|\;\exists t_k\downarrow
0,\;w_k\to w\;\mbox{ with }\;\bar z+t_k w_k\in\Omega ~\forall ~ k\}.
\end{equation*}
The (Fr\'{e}chet) {\em regular normal cone} and the (Mordukhovich) {\em limiting/basic normal cone} to $\Omega$ at $\bar
z$ are  defined by
\begin{eqnarray}
&& \widehat N_\Omega(\bar z):=(T_\Omega(\bar z))^\ast\nonumber\\
\mbox{and }  &&
N_\Omega(\bar z):=\left \{z^\ast \mv \exists z_{k}\stackrel{\Omega}{\to}\zb \mbox{ and } z^\ast_k\rightarrow z^\ast \mbox{ such that } z^\ast_{k}\in \widehat{N}_{\Omega}(z_k) \  \forall k \right \}
\nonumber
\end{eqnarray}
respectively.\\
Further, if $\zb\not\in\Omega$ we define
\[T_\Omega(\zb):=\widehat N_\Omega(\zb):=N_\Omega(\zb):=\emptyset.\]
\end{definition}
When  the set $\Omega$ is convex, the tangent/contingent cone and the regular/limiting normal cone reduce to
the classical tangent cone and normal cone of convex analysis respectively. The regular normal cone $\widehat N_\Omega(\zb)$ is always convex whereas the limiting normal cone can be non-convex  if $\Omega$ is not convex.

\begin{lemma}\label{LemInclTangCone} Let $\Omega\subseteq\R^d$  be closed and $\zb\in\Omega$. Then
\begin{equation}\label{EqLimNormalDir2}N_\Omega(\zb)\supseteq N_{T_\Omega(\zb)}(0)=\widehat N_\Omega(\zb)\cup\bigcup_{0\not=w\in\R^d}N_{T_\Omega(\zb)}(w).\end{equation}
\end{lemma}
\begin{proof}
   The inclusion $N_\Omega(\zb)\supseteq N_{T_\Omega(\zb)}(0)$ in \eqref{EqLimNormalDir2} was shown in \cite[Proposition 6.27]{RoWe98}. It also follows from
  \cite[Proposition 6.27]{RoWe98} together with $N_{T_\Omega(\zb)}(0)\supseteq \widehat N_{T_\Omega(\zb)}(0)=\widehat N_\Omega(\zb)$ that $N_{T_\Omega(\zb)}(0)\supseteq\widehat N_\Omega(\zb)\cup\bigcup_{0\not=w\in\R^d}N_{T_\Omega(\zb)}(w)$. In order to show the reverse inclusion, consider $w^\ast\in N_{T_\Omega(\zb)}(0)$ together with sequences $w_k\to 0$, ${w_k}^\ast\to w^\ast$ with $w_k^\ast\in \widehat N_{T_\Omega(\zb)}(w_k)$ $\forall k$. If $w_k=0$ holds for infinitely many $k$, then $w^\ast \in \widehat N_{T_\Omega(\zb)}(0)=\widehat N_\Omega(\zb)$ follows because $\widehat N_\Omega(\zb)$ is closed.
On the other hand, if $w_k\not=0$ holds for all but finitely many $k$ by passing to a subsequence we can assume that $w_k/\norm{w_k}$ converges to some $w$, and because of $\widehat N_{T_\Omega(\zb)}(w_k)=\widehat N_{T_\Omega(\zb)}(w_k/\norm{w_k})$ we conclude $w^\ast\in N_{T_\Omega(\zb)}(w)$. Hence \eqref{EqLimNormalDir2} is established and this finishes the proof.
\end{proof}

Usually, the computation of the  limiting normal cone to a nonconvex set $\Omega$ is a difficult task. A special case when the limiting normal cone has a comparatively simple description is given by polyhedral sets.

\begin{definition}Let $\Omega\subseteq\R^d$.
\begin{enumerate}
\item We say that $\Omega$ is {\em convex polyhedral}, if the set can be written as the intersection of finitely many halfspaces, i.e. there are elements $(a_i,\alpha_i)\in\R^d\times\R$, $i=1,\ldots,p$ such that $\Omega=\{z\mv \skalp{a_i,z}\leq\alpha_i,\ i=1,\ldots,p\}$.
\item We say that $\Omega$ is {\em polyhedral}, if it is the union of finitely many convex polyhedral sets.
\item Given a point $\zb\in\Omega$, we say that $\Omega$ is {\em locally polyhedral near $\zb$} if there is a neighborhood $W$ of $\zb$ and a polyhedral set $C$ such that $\Omega\cap W= C\cap W$.
\end{enumerate}
\end{definition}

\begin{lemma}
  Let $\Omega\subseteq\R^d$ be locally polyhedral near some point $\zb\in \Omega$. Then
  \begin{equation}
    \label{EqLimNormalDirPoly}N_\Omega(\zb)=\bigcup_{w\in T_\Omega(\zb)}\widehat N_{T_\Omega(\zb)}(w).
  \end{equation}
\end{lemma}
\begin{proof}
  Follows from \cite[Lemma 2.2]{Gfr14a}.
\end{proof}

In this paper the notion of {\em metric subregularity} will play an important role.
\begin{definition}
  	 Let $M:\R^d\rightrightarrows\R^s$ be a set-valued mapping and let $(\zb,\wb)\in\Gr M$. We say that $M$ is {\em metrically subregular} at $(\zb,\wb)$ if there exist a neighborhood $W$ of $\zb$ and a positive number  $\kappa>0$ such that
  \begin{equation}\label{EqMetrSubReg}\dist{z,M^{-1}(\wb)}\leq\kappa\dist{\wb,M(z)}\ \; \forall z\in W.
  \end{equation}
\end{definition}
It is well-known that metric subregularity of $M$ at $(\zb,\wb)$ is equivalent with the property of {\em calmness} of the inverse mapping $M^{-1}$ at $(\wb,\zb)$, cf. \cite{DoRo04}. Further,
metric subregularity of $M$ at $(\zb,\wb)$ is equivalent with metric subregularity of the mapping $z\to (z,\bar w)-\Gr M$ at $(\zb,(0,0))$, cf. \cite[Proposition 3]{GfrYe17a}.
\begin{lemma}
  \label{LemConicalMult}Let $M:\R^d\rightrightarrows\R^s$ be a set-valued mapping, let $(\zb,\wb)\in\Gr M$ and assume that $\Gr M$ is a closed cone. If $M$ is metrically subregular at $(0,0)$ then there is some $\kappa>0$ such that
  \[\dist{z,M^{-1}(0)}\leq\kappa\dist{0,M(z)}\  \forall z\in\R^d.\]
  In particular, $M$ is metrically subregular at every point $(\zb,0)\in\Gr M$.
\end{lemma}
\begin{proof}According to the definition of metric subregularity, consider a neighborhood $W$ of $0$ and a real $\kappa>0$ such that $\dist{z,M^{-1}(0)}\leq \kappa \dist{0,M(z)}$ for all $z\in W$. Now consider $z\in\R^d$. Then we can find some $\lambda>0$ such that $\lambda z\in W$ and thus $\dist{\lambda z,M^{-1}(0)}\leq \kappa \dist{0,M(\lambda z)}$. Since $\Gr M$ is a cone it follows that $M^{-1}(0)$ is a cone and $M(\lambda z)=\lambda M(z)$. Hence $\lambda\dist{z,M^{-1}(0)}=\dist{\lambda z,M^{-1}(0)}\leq \kappa \dist{0,M(\lambda z)}=\lambda \kappa \dist{0,M(z)}$ and $\dist{z,M^{-1}(0)}\leq\kappa\dist{0,M(z)}$ follows.
\end{proof}

The following lemma is a special variant of \cite[Proposition 2.1]{Gfr11}.
\begin{lemma}\label{LemSubRegLinear}Let $P:\R^d\to\R^s$ be contiunously differentiable, let $D\subseteq\R^s$ be closed and assume that the mapping $z\rightrightarrows P(z)-D$ is metrically subregular at $(\zb,0)$. Then the mapping $u\rightrightarrows \nabla P(\zb)u-T_D(P(\zb))$ is metrically subregular at $(0,0)$.
\end{lemma}
Given a cone $C\subseteq \R^d$, we denote by $\Lsp(C)$ the largest subspace $L\subseteq \R^d$ such that
\[C+L\subseteq C.\]
Note that $\Lsp(C)$ is well defined because for two subspaces $L_1,L_2$ fulfilling $C+L_i\subseteq C$, $i=1,2$ we have
\begin{equation}\label{EqLineality1}C+L_1+L_2=(C+L_1)+L_2\subseteq C+L_2\subseteq C\end{equation}
and we are working in finite dimensional spaces. Note that for every subspace $L$ we have $C+L\supseteq C$ and thus $C+\Lsp(C)=C$.
If $C$ is a convex cone, then $\Lsp(C)=C\cap (-C)$ is the so-called {\em lineality space} of $C$, the largest subspace contained in $C$.
\begin{lemma}\label{LemLineality}
  Let $C\subseteq \R^d$ be a closed cone and let $\zb\in C$. Then
  \[\Lsp(C)+[\zb]\subseteq \Lsp(T_C(\zb)).\]
\end{lemma}
\begin{proof}
  We show that both $T_C(\zb)+\Lsp(C)\subseteq T_C(\zb)$ and $T_C(\zb)+[\zb]\subseteq T_C(\zb)$. Then the statement follows from \eqref{EqLineality1}. Consider a tangent $w\in T_C(\zb)$ together with sequences $t_k\downarrow 0$ and $w_k\to w$ with $\zb+t_kw_k\in C$ for all $k$. For fixed $l\in \Lsp(C)$ and  for every $k$ we have $t_kl\in \Lsp(C)$ and thus $\zb+t_kw_k+t_kl=\zb+t_k(w_k+l)\in C$. Hence $w+l\in T_C(\zb)$ and $T_C(\zb)+\Lsp(C)\subseteq T_C(\zb)$ follows. Next, let $\gamma\in\R$. By passing to a subsequence we can assume  $1+t_k\gamma>0$ and thus
  \[(1+t_k\gamma)(\zb+t_kw_k)=\zb +t_k(1+t_k\gamma)\left(w_k+\frac\gamma{1+t_k\gamma}\zb\right)\in C\ \forall k.\]
  Since $t_k(1+t_k\gamma)\downarrow 0$ and $w_k+\frac\gamma{1+t_k\gamma}\zb\to w+\gamma\zb$, we conclude $w+\gamma\zb\in T_C(\zb)$ and the second claimed inclusion $T_C(\zb)+[\zb]\subseteq T_C(\zb)$ follows. This finishes the proof.
\end{proof}
At the end of this section we recall the definition of the critical cone to a set.
\begin{definition}\label{DefCritCone}
Given a  set $\Omega$ and an element $\zb\in\Omega$ together with a regular normal $\zba\in \widehat N_\Omega(\zb)$ we define the {\em critical cone} to $\Omega$ at $(\zb,\zba)$ as
\[\K_{\Omega}(\zb,\zba):=T_\Omega(\zb)\cap [\zba]^\perp.\]
\end{definition}

\section{Stationarity concepts\label{SecStat}}
In this section we recall some basic fact about  stationarity concepts for the general problem \eqref{EqGenOptProbl}.

We denote by $\Omega$ the feasible region of the problems \eqref{EqGenOptProbl}, i.e.
\begin{eqnarray}
\label{EqOmega}  \Omega&:=&\{z\in\R^d\mv P(z)\in D\}
\end{eqnarray}
Further, given $\zb\in\Omega$ we denote by
\[\TlinO(\zb):=\{u\in\R^d\mv \nabla P(\zb)u\in T_D(P(\zb))\}\]
the {\em linearized tangent cone to $\Omega$ at $\zb$}. Recall that there always holds
\begin{equation}\label{EqInclGACQ}
  T_\Omega(\zb)\subseteq \TlinO(\zb).
\end{equation}
We use the notation $\TlinO(\zb)$ to indicate that the linearized tangent cone depends on $P$ and $D$, i.e., if we have two equivalent representations
\begin{equation}\label{EqEquivRepres}\Omega=\{z\mv P_1(z)\in D_1\}=\{z\mv P_2(z)\in D_2\}\end{equation}
with continuously differentiable mappings $P_i:\R^d\to\R^{s_i}$ and closed sets $D_i\subseteq R^{s_i}$, $i=1,2$,
then we can have $\Tlin_{P_1,D_1}(\zb)\not=\Tlin_{P_2,D_2}(\zb)$.
\begin{definition}
  Let $\zb\in\Omega$. We say that $\zb$ is
  \begin{enumerate}
    \item {\em B-stationary (Bouligand stationary)} for the problem \eqref{EqGenOptProbl}, if
    \[0\in \nabla f(\zb)+\widehat N_\Omega(\zb),\]
    \item {\em S-stationary (strong stationary)} for the problem \eqref{EqGenOptProbl}, if
    \begin{eqnarray*}0\in \nabla f(\zb)+\nabla P(\zb)^\ast\widehat N_D(P(\zb)),\end{eqnarray*}
    \item {\em M-stationary} for the problem \eqref{EqGenOptProbl}, if
    \begin{eqnarray*}0\in \nabla f(\zb)+\nabla P(\zb)^\ast N_D(P(\zb)).\end{eqnarray*}

  \end{enumerate}
\end{definition}
Note that S- and M-stationarity depend on $P$ and $D$ used for describing of $\Omega$ whereas B-stationarity  is independent of the representation of $\Omega$.

B-stationarity can be equivalently expressed as
\[\skalp{\nabla f(\zb),w}\geq 0\ \forall w\in T_\Omega(\zb).\]
By saying that a {\em feasible descent direction} for the program \eqref{EqGenOptProbl} at $\zb$ is a direction $w\in T_\Omega(\zb)$ with $\skalp{\nabla f(\zb),w}<0$, we see that B-stationarity conveys the fact that no feasible descent direction exists.
It is well known that every local minimizer is also B-stationary, cf. \cite[Theorem 6.12]{RoWe98}. Conversely, if $\zb$ is B-stationary for the program \eqref{EqGenOptProbl}, then by \cite[Theorem 6.11]{RoWe98} there exists a smooth mapping $\tilde f:\R^d\to\R$ such that $\tilde f(\zb)=f(\zb)$, $\nabla \tilde f(\zb)=\nabla f(\zb)$ and $\zb$ is a global minimizer of the program
\[\min_z\tilde f(z)\quad\mbox{subject to}\quad P(z)\in D.\]
Thus, if the available first-order information at the point $\zb$ is provided solely by $T_\Omega(\zb)$ and $\nabla f(\zb)$, then B-stationarity constitutes the best possible first-order optimality condition and thus characterizing B-stationarity is the primary goal.

However, the computation of the regular normal cone $\widehat N_\Omega(\zb)$ appearing in the definition of B-stationarity can be a very difficult task for general sets $D$ and therefore, besides other stationary concepts, the notions of S- and M-stationarity have been introduced. S-stationarity was first considered in the  monograph by Luo, Pang and Ralph \cite{LuPaRa96} whereas M-stationarity conditions appeared first in the papers by Outrata \cite{Out99} and Ye \cite{Ye99}, respectively. The monikers M-stationarity and S-stationarity were coined in \cite{Sch00,SchSch00} for MPCC and then carried over in \cite{FleKanOut07} to the general problem \eqref{EqGenOptProbl}.

By applying \cite[Theorem 6.14]{RoWe98} we readily obtain the inclusion
  \begin{equation}\label{EqInclRegNormal}\widehat N_\Omega(\zb) \supseteq \nabla P(\zb)^\ast\widehat N_D(P(\zb)).\end{equation}
Hence we deduce from the definition that S-stationarity of $\zb$ implies B-stationarity. However, the reverse implication is only valid under comparatively strong assumptions. We state here the following result due to Gfrerer and Outrata \cite[Theorem 4]{GfrOut16b}.
\begin{theorem}\label{ThSuffS_Stat}Assume that $\zb$ is feasible for the problem \eqref{EqGenOptProbl},  assume that the mapping $z\rightrightarrows P(z)-D$ is metrically subregular at $(\zb,0)$ and assume that
  \[ \nabla P(\zb)\R^d +\Lsp(T_D(P(\zb)))=\R^s.\]
  Then \eqref{EqInclRegNormal} holds with equality. In particular, if $\zb$ is B-stationary then it is  S-stationary as well.
\end{theorem}
It is well known that B-stationarity implies M-stationarity under mild constraint qualification conditions.
\begin{definition}
  Let  $P(\zb)\in D$.
  \begin{enumerate}
  \item (cf. \cite{FleKanOut07}) We say that the {\em generalized Abadie constraint qualification} (GACQ) holds at $\zb$ if
    \begin{equation}
      \label{EqGACQ}T_\Omega(\zb)=\TlinO(\zb).
    \end{equation}
  \item (cf. \cite{FleKanOut07}) We say that the {\em generalized Guignard constraint qualification} (GGCQ) holds at $\zb$ if
    \begin{equation}
      \label{EqGGCQ}\widehat N_\Omega(\zb)=\big(\TlinO(\zb)\big)^\ast.
    \end{equation}
    \item (cf. \cite{GfrMo15a}) We say that the {\em metric subregularity constraint qualification (MSCQ)} holds at $\zb$  if the set-valued map $M(z):=P(z)-D$ is metrically subregular at $(\zb,0)$.
    \end{enumerate}
\end{definition}

We always have
\[\mbox{MSCQ}\ \Longrightarrow\ \mbox{GACQ}\ \Longrightarrow\ \mbox{GGCQ}.\]
Indeed, the first implication follows from \cite[Proposition 1]{HenOut05} whereas the second implication obviously holds true. Note that all these constraint qualifications depend on the representation of $\Omega$ by $P$ and $D$. GGCQ seems to be indispensable for verifying B-stationarity solely with first-order derivatives of the problem functions.

We state here the following result from the recent paper by Benko and Gfrerer \cite[Proposition 3]{BeGfr17a}.
\begin{theorem}\label{ThLinM_Stat}Assume that $\zb$ is feasible for the problem \eqref{EqGenOptProbl} and assume that GGCQ is fulfilled, while the mapping $u\rightrightarrows \nabla P(\zb)u- T_D(P(\zb))$ is metrically subregular at $(0,0)$.
Then
\begin{equation}
  \label{EqUpperInclRegNormal}\widehat N_\Omega(\zb)\subseteq \nabla P(\zb)^\ast N_{T_D(P(\zb))}(0) \subseteq \nabla P(\zb)^\ast N_D(P(\zb)).
\end{equation}
\end{theorem}
\begin{remark}\label{RemSubReg}
   Note that the assumptions of Theorem \ref{ThLinM_Stat} are fulfilled if MSCQ holds at $\zb$. Indeed, MSCQ implies GGCQ and metric subregularity of $u\rightrightarrows \nabla P(\zb)u- T_D(P(\zb))$ at $(0,0)$ follows from Lemma \ref{LemSubRegLinear}.
\end{remark}

If $\zb$ is B-stationary and the assumptions of Theorem \ref{ThLinM_Stat} are fulfilled, it follows from the second inclusion in \eqref{EqUpperInclRegNormal} and the definition that $\zb$ is M-stationary. Other constraint qualifications ensuring M-stationarity can be found in \cite{Ye05}. However, from the first inclusion in \eqref{EqUpperInclRegNormal} we also derive the necessary optimality condition
\begin{equation}\label{EqLinMStat}0\in \nabla f(\zb)+ \nabla P(\zb)^\ast N_{T_D(P(\zb))}(0)\end{equation}
and this is stronger than M-stationarity because we always have
\[N_{T_D(P(\zb))}(0)\subseteq N_D(P(\zb))\]
by \cite[Proposition 6.27]{RoWe98}.

\section{Linearized M-stationarity conditions\label{SecLM_stat}}

One of the  basic statements of this section is provided by the following proposition, which can be considered as a refinement of the necessary condition \eqref{EqLinMStat}.

\begin{proposition}\label{PropStrongFirstOrder}
  Let $\zb$ be B-stationary for the optimization problem \eqref{EqGenOptProbl} and assume that GGCQ is fulfilled, while the mapping $u\rightrightarrows \nabla P(\zb)u- T_D(P(\zb))$ is metrically subregular at $(0,0)$. Then one of the following two conditions is fulfilled:
  \begin{enumerate}
    \item There is $w\in T_D(P(\zb))$  and a multiplier  $w^\ast \in\widehat N_{T_D(P(\zb))}(w)$ such that
  \begin{equation}
    \label{EqKKT}\nabla f(\zb) +\nabla P(\zb)^\ast   w^\ast  =0.
  \end{equation}
    \item There is $\ub\in \TlinO(\bar z)$ such that
    \begin{eqnarray}
\label{EqNotZero}&&\nabla P(\zb)\ub\not\in \Lsp(T_D(P(\zb))),\\
\label{EqKKTCritDir}      &&\skalp{\nabla  f(\zb), \ub}=0,\\
\label{EqKKTNormal}      &&0\in \nabla f(\zb)+\widehat N_{\TlinO(\bar z)}(\ub)
    \end{eqnarray}
    and $T_D(P(\zb))$ is  not locally polyhedral near $\nabla P(\zb)\ub$.
  \end{enumerate}
\end{proposition}
Before proving this theorem we discuss some of its issues.
We will call a direction $u\in\TlinO(\zb)$ satisfying \eqref{EqKKTCritDir} a {\em critical direction} for the problem \eqref{EqGenOptProbl}. Now assume that the first statement of Proposition \ref{PropStrongFirstOrder} fails to hold and thus there exist $\ub$ fulfilling the second statement. Let us rename $\ub$ by $u_1$.
From \eqref{EqNotZero} it follows that $\nabla P(\zb)u_1\not=0$ and thus $u_1\not=0$ as well. Further, since $u_1$ is a critical direction and $\zb$ is assumed to be B-stationary for the problem \eqref{EqGenOptProbl}, it follows that $u_1$ is a global solution of the program
\begin{equation}\label{EqLinProbl}\min \skalp{\nabla f(\zb),u}\quad\mbox{subject to}\quad\nabla P(\zb)u\in T_D(P(\zb))\end{equation}
and \eqref{EqKKTNormal} is the corresponding B-stationarity condition. This is not really surprising, but the important point is that we can apply Proposition \ref{PropStrongFirstOrder} once more to the problem \eqref{EqLinProbl} at $u_1$. Indeed, since the mapping $u\rightrightarrows \nabla P(\zb)u- T_D(P(\zb))$ is assumed to be metrically subregular at $(0,0)$ and its graph is a closed cone, by Lemma \ref{LemConicalMult} it is metrically subregular at $(u_1,0)$ as well. By taking into account Remark \ref{RemSubReg} we see that GGCQ holds for the system $\nabla P(\zb)u\in T_D(P(\zb))$ at $u_1$ and the linearized mapping $u\rightrightarrows \nabla P(\zb)u- T_{T_D(P(\zb))}(\nabla P(\zb)u_1)$ is metrically subregular at $(0,0)$. Thus we can apply
Proposition \ref{PropStrongFirstOrder} to obtain either the existence of some direction $w\in  T_{T_D(P(\zb))}(\nabla P(\zb)u_1)$ and some multiplier $w^\ast  \in \widehat N_{T_{T_D(P(\zb))}(\nabla P(\zb)u_1)}(w)$ such that \eqref{EqKKT} holds or the existence of some direction $u_2\in \TlinOk{1}(\zb;u_1):=\{u\mv \nabla P(\zb)u\in T_{T_D(P(\zb))}(\nabla P(\zb)u_1)\}$ such that
    \begin{eqnarray*}
&&\nabla P(\zb)u_2\not\in \Lsp(T_{T_D(P(\zb))}(\nabla P(\zb)u_1)),\\
&&\skalp{\nabla  f(\zb), u_2}=0,\\
&&0\in \nabla f(\zb)+\widehat N_{\Tlin_{\nabla P(\zb),T_D(P(\zb))}(u_1)}(u_2)
    \end{eqnarray*}
and $T_{T_D(P(\zb))}(\nabla P(\zb)u_1)$ is  not locally polyhedral near $\nabla P(\zb)u_2$. Again, if the first case does not emerge we can repeat the procedure.
Let us recursively define for $\yb\in D$ and directions $v_1,v_2,\ldots$ the following $k$-th order tangent cones to $D$ by
\[T^0_D(\yb):=T_D(\yb),\ T^k_D(\yb;v_1,\ldots,v_k):=T_{T^{k-1}_D(\yb;v_1,\ldots,v_{k-1})}(v_k),\ k\geq 1.\]
Note that by the definition of the tangent cone we have $T^k_D(\yb;v_1,\ldots,v_k)=\emptyset$ if $v_k\not\in T^{k-1}_D(\yb;v_1,\ldots,v_{k-1})$.
Then we can also define the following $k$-th order linearized tangent cones to $\Omega$ by
\begin{align*}&\TlinOk{0}(\zb)=\TlinO(\zb),\\
&\TlinOk{k}(\zb;u_1,\ldots,u_k):=\{u\mv \nabla P(\zb)u\in T^k_D(P(\zb);\nabla P(\zb)u_1,\ldots,\nabla P(\zb)u_k)\},\ k\geq 1.
\end{align*}
When we apply Proposition \ref{PropStrongFirstOrder} the k-th time we find either  a direction
\[w\in T^{k-1}_D(P(\zb);\nabla P(\zb)u_1,\ldots,\nabla P(\zb)u_{k-1})\]
together with a multiplier
\[w^\ast  \in\widehat N_{T^{k-1}_D(P(\zb);\nabla P(\zb)u_1,\ldots,\nabla P(\zb)u_{k-1})}(w)\] such that $\nabla f(\zb)+\nabla P(\zb)^\ast  w^\ast  =0$ or a direction $u^k\in \TlinOk{k-1}(\zb;u_1,\ldots,u_{k-1})$ such that
\begin{align}
\label{EqAux1}&\nabla P(\zb)u_k\not\in \Lsp(T^{k-1}_D(P(\zb);\nabla P(\zb)u_1,\ldots,\nabla P(\zb)u_{k-1})),\\
&\skalp{\nabla  f(\zb), u_k}=0,\\
&0\in \nabla f(\zb)+\widehat N_{\TlinOk{k-1}(\zb;u_1,\ldots,u_{k-1})}(u_k)
\end{align}
and $T^{k-1}_D(P(\zb);\nabla P(\zb)u_1,\ldots,\nabla P(\zb)u_{k-1})$ is  not locally polyhedral near $\nabla P(\zb)u_k$. Next observe that we cannot infinitely often apply Proposition \ref{PropStrongFirstOrder}. By Lemma \ref{LemLineality} we have
\begin{align*}\lefteqn{\Lsp(T^k_D(P(\zb);\nabla P(\zb)u_1,\ldots,\nabla P(\zb)u_k))}\\
&\supseteq\Lsp(T^{k-1}_D(P(\zb);\nabla P(\zb)u_1,\ldots,\nabla P(\zb)u_{k-1}))+[\nabla P(\zb)u_k]\end{align*}
and together with \eqref{EqAux1} we obtain
\begin{align*}\lefteqn{\dim \Lsp(T^k_D(P(\zb);\nabla P(\zb)u_1,\ldots,\nabla P(\zb)u_k))}\\
&\geq \dim\Lsp(T^{k-1}_D(P(\zb);\nabla P(\zb)u_1,\ldots,\nabla P(\zb)u_{k-1}))+1.
\end{align*}
Since we work in finite dimensions the finiteness of $k$ follows. Summing up we have  shown the following theorem.
\begin{theorem}
  \label{ThStrongFirstOrder}  Let $\zb$ be B-stationary for the optimization problem \eqref{EqGenOptProbl} and assume that GGCQ is fulfilled, while the mapping $u\rightrightarrows \nabla P(\zb)u- T_D(P(\zb))$ is metrically subregular at $(0,0)$. Then there exists a natural number $k\geq 0$, directions $u_1,\ldots,u_k$ and $w\in T^k_D(P(\zb);\nabla P(\zb)u_1,\ldots,\nabla P(\zb)u_k)$ and a multiplier
  $w^\ast  \in\widehat N_{T^k_D(P(\zb);\nabla P(\zb)u_1,\ldots,\nabla P(\zb)u_k)}(w)$ such that
  \[\nabla f(\zb)+\nabla P(\zb)^\ast  w^\ast  =0.\]
  Moreover, for every $l=1,\ldots,k$ we have
  \begin{align}
\label{Eq_u_l1}    &u_l\in \TlinOk{l-1}(\zb;u_1,\ldots,u_{l-1})\\
\label{Eq_u_l2}    &\nabla P(\zb)u_l\not\in \Lsp(T^{l-1}_D(P(\zb);\nabla P(\zb)u_1,\ldots,\nabla P(\zb)u_{l-1})),\\
\label{Eq_u_l3}    &\skalp{\nabla f(\zb),u_l}=0
  \end{align}
  and $T^{l-1}_D(P(\zb);\nabla P(\zb)u_1,\ldots,\nabla P(\zb)u_{l-1})$ is  not locally polyhedral near $\nabla P(\zb)u_l$.
\end{theorem}
It is easy to see that Theorem \ref{ThStrongFirstOrder} considerably strengthen the necessary optimality condition \eqref{EqLinMStat}, which in turn is stronger than the usual M-stationary condition.  As candidates for the multipliers $w^\ast  $ fulfilling the first-order optimality condition \eqref{EqKKT1}
we consider multipliers fulfilling
\begin{equation}\label{EqMultIncl}w^\ast   \in \widehat N_{T^k_D(P(\zb),\nabla P(\zb)u_1,\ldots,\nabla P(\zb)u_k)}(w)\end{equation}
for some $w\in T^k_D(P(\zb),\nabla P(\zb)u_1,\ldots,\nabla P(\zb)u_k)$, where the directions $u_l$, $l=1,\ldots,k$ fulfill the conditions of Theorem \ref{ThStrongFirstOrder}. By applying the following lemma we immediately obtain that the set on the right hand side of the inclusion \eqref{EqMultIncl} is contained in $N_{T_D(P(\zb))}(0)\subseteq N_D(P(\zb))$.

\begin{lemma}Let $\yb\in D$. Then for every collection of directions $v_1,\ldots,v_l\in\R^s$ we have
\[\widehat N_{T^{l-1}_D(\yb;v_1,\ldots,v_{l-1})}(v_l)\subseteq N_{T^{l-1}_D(\yb;v_1,\ldots,v_{l-1})}(v_l)\subseteq N_{T_D(\yb)}(0)\subseteq N_D(\yb).\]
\end{lemma}
\begin{proof}
We will show the lemma by induction with respect to the number of directions $l$. Indeed, for $l=1$ the claimed inclusions  hold true because for all $v_1$ we have
$\widehat N_{T_D^0(\yb)}(v_1)\subseteq N_{T_D^0(\yb)}(v_1)\subseteq N_{T_D(\yb)}(0)\subseteq N_D(\yb)$ by the definitions of the regular/limiting normal cone and \eqref{EqLimNormalDir2}. Now assume that the claim holds true for some number $l\geq 1$ and consider arbitrary directions $v_1,\ldots,v_{l+1}$. Then by the definitions of the regular/limiting normal cone, \eqref{EqLimNormalDir2} and the induction hypothesis we obtain
\begin{align*}
  \widehat N_{T^l_D(\yb;v_1,\ldots,v_l)}(v_{l+1})&\subseteq N_{T^l_D(\yb;v_1,\ldots,v_l)}(v_{l+1})= N_{T_{T^{l-1}_D(\yb;v_1,\ldots,v_{l-1})}(v_l)}(v_{l+1})\\
  &\subseteq N_{T_{T^{l-1}_D(\yb;v_1,\ldots,v_{l-1})}(v_l)}(0)\subseteq N_{T^{l-1}_D(\yb;v_1,\ldots,v_{l-1})}(v_l)\\
  &\subseteq N_{T_D(\yb)}(0)\subseteq N_D(\yb)
\end{align*}
and the lemma is proved.
\end{proof}
We do not know so much about the order $k$ appearing in Theorem \ref{ThStrongFirstOrder}. By using \eqref{Eq_u_l2} and Lemma \ref{LemLineality}, a rough upper estimate for $k$ is given by $\dim(\nabla P(\zb)\R^d)-\dim(\Lsp(T_D(P(\zb))\cap \nabla P(\zb)\R^d)$. However, in many examples we found that this bound is too pessimistic and the necessary optimality conditions of Theorem \ref{ThStrongFirstOrder} hold with small $k$, say $k=0,1$ or $2$. More research has to be done to investigate this circumstance.

Recall that a local minimizer $\zb$ for \eqref{EqGenOptProbl} is called a {\em sharp minimum} if there is a constant $\alpha>0$ such that
\[f(z)\geq f(\zb)+\alpha\norm{z-\zb}\]
holds for all feasible $z$ close to $\zb$.
\begin{lemma}\label{LemSharpMin}
  Assume that at $\zb$ GGCQ is fulfilled. Then $\zb$ is a sharp minimum if and only if there is some $\alpha'>0$ such that
  \begin{equation}\label{EqSharpMin}\skalp{\nabla f(\zb),u}\geq \alpha'\norm{u}\ \forall u\in\TlinO(\zb).\end{equation}
\end{lemma}
\begin{proof}In order to show the sufficiency of \eqref{EqSharpMin} for $\zb$ being a sharp minimum, assume on the contrary that there is a sequence $z_k$ of feasible points converging to $\zb$ satisfying
\[\liminf_{k\to\infty}\frac{f(z_k)-f(\zb)}{\norm{z_k-\zb}}=\liminf_{k\to\infty} \skalp{\nabla f(\zb),\frac{z_k-\zb}{\norm{z_k-\zb}}}\leq 0\]
By passing to a subsequence we can assume that $\frac{z_k-\zb}{\norm{z_k-\zb}}$ converges to some $u$. Then $\skalp{\nabla f(\zb),u}\leq 0$ and $u\in T_\Omega(\zb)\subset \TlinO(\zb)$ contradicting \eqref{EqSharpMin}. To prove necessity of \eqref{EqSharpMin}, assume that $\zb$ is a sharp minimum and consider a tangent $u\in T_\Omega(\zb)$ together with sequences $t_k\downarrow 0$ and $u_k\to u$ satisfying $P(\zb+t_ku_k)\in D$. Then
\[f(\zb+t_ku_k)-f(\zb)=t_k\skalp{\nabla f(\zb),u_k}+\oo(t_k\norm{u_k})\geq \alpha t_k\norm{u_k}\]
and by dividing by $t_k$ and passing to the limit we obtain $\skalp{\nabla f(\zb),u}\geq \alpha\norm{u}$. Next consider $u\in \co T_\Omega(\zb)$ together with elements $u_1,\ldots u_K$ and positive scalars $\gamma_1,\ldots,\gamma_K$, $\sum_{i=1}^K\gamma_i=1$ such that $u=\sum_{i=1}^K\gamma_iu_i$. Then
\[\skalp{\nabla f(\zb), u}=\sum_{i=1}^K\gamma_i\skalp{\nabla f(\zb),u_i}\geq \alpha\sum_{i=1}^K\gamma_i\norm{u_i}\geq \alpha \norm{\sum_{i=1}^K\gamma_iu_i}=\alpha \norm{u}\]
and we easily conclude
\[\skalp{\nabla f(\zb), u}\geq \alpha \norm{u}\ \forall u\in\cl\co T_\Omega(\zb).\]
By dualizing \eqref{EqGGCQ} we have $\cl\co T_\Omega(\zb)=\cl\co \TlinO(\zb)$ and \eqref{EqSharpMin} follows.
\end{proof}

\begin{corollary}\label{CorSharpMin}
  Assume that $\zb$ is a sharp minimum for \eqref{EqGenOptProbl} and assume that GGCQ is fulfilled, while the mapping $u\rightrightarrows \nabla P(\zb)u- T_D(P(\zb))$ is metrically subregular at $(0,0)$. Then there is $w\in T_D(P(\zb))$  and a multiplier  $w^\ast \in\widehat N_{T_D(P(\zb))}(w)$ such that $\nabla f(\zb) +\nabla P(\zb)^\ast   w^\ast  =0$.
\end{corollary}
\begin{proof}
  The statement follows immediately from Proposition \ref{PropStrongFirstOrder}, because by Lemma \ref{LemSharpMin} the second alternative of Proposition \ref{PropStrongFirstOrder} is not possible.
\end{proof}
Note that the conclusion of Corollary \ref{CorSharpMin} can also hold  in situations when $\zb$ is not a sharp minimum. Besides the cases when there does not exist a direction $\bar u$ fulfilling the conditions of the second alternative of Proposition \ref{PropStrongFirstOrder}, the first alternative of Proposition \ref{PropStrongFirstOrder} holds true if there exists some direction $\bar u$ satisfying $\skalp{\nabla f(\zb),\bar u}=0$, $\nabla P(\zb)\bar u\in T_D(P(\zb))$ such that $\bar u$ is an S-stationary solution of \eqref{EqLinProbl} because then $0\in\nabla f(\zb)+\nabla P(\zb)^\ast\widehat N_{T_D(P(\zb))}(\nabla P(\zb)\bar u)$ by the definition of S-stationarity. By Theorem \ref{ThSuffS_Stat} we know that the condition \[ \nabla P(\zb)\R^d +\Lsp(T_{T_D(P(\zb))})(\nabla P(\zb)\bar u)=\R^s\]
is sufficient for S-stationarity of $\bar u$ and since $\Lsp(T_{T_D(P(\zb))})(\nabla P(\zb)\bar u)$ is always larger than $\Lsp(T_D(P(\zb)))$ it is possible that such an S-stationary solution $\bar u$ of \eqref{EqLinProbl} exists even if $\zb$ is not S-stationary for \eqref{EqGenOptProbl}.

We now turn to the proof of Proposition \ref{PropStrongFirstOrder}. At first we need some prerequisites. As introduced in the recent paper by Benko and Gfrerer \cite{BeGfr16d}, consider the program
  \begin{equation}\label{EqAuxProgr}\min_{(u,y)\in\R^d\times\R^s}\skalp{\nabla f(\zb),u} +\frac12 \norm{y}^2\quad \mbox{subject to}\quad\nabla P(\zb)u+y\in T_D(P(\zb)).
  \end{equation}
\begin{lemma}\label{LemAuxProbl}Assume that the assumptions of Proposition  \ref{PropStrongFirstOrder} are fulfilled. Then  MSCQ holds for the system $\nabla P(\zb)u+y\in T_D(P(\zb))$ at every point $(\bar u,\yb)$ feasible   for \eqref{EqAuxProgr}. Further, the program \eqref{EqAuxProgr} is bounded below and every B-stationary solution $(\bar u,\yb)$ is also S-stationary, i.e. there is some multiplier $w^\ast  \in \widehat N_{T_D(P(\zb))}(\nabla P(\zb)\bar u+\yb)$ such that
\begin{equation}\label{EqSStatAuxProbl}\nabla f(\zb)+\nabla P(\zb)^\ast  w^\ast  =0,\quad \yb+w^\ast  =0.\end{equation}
\end{lemma}
\begin{proof}Consider the set-valued mapping $M(u,y):=\nabla P(\zb)u+y-T_D(P(\zb))$. Given any $(u,y)\in\R^d\times \R^s$ we can find $v\in M(u,y)$ such that $\norm{v}=\dist{0,M(u,y)}$ because $M(u,y)$ is closed. Then $0\in M(u,y-v)$ showing that
\[\dist{(u,y),M^{-1}(0)}\leq \norm{v}=\dist{0,M(u,y)}\]
and MSCQ for the system $\nabla P(\zb)u+y\in T_D(P(\zb))$ at every point $(\bar u,\yb)$ feasible for \eqref{EqAuxProgr} follows.
In order to show the boundedness of the program \eqref{EqAuxProgr} assume on the contrary that \eqref{EqAuxProgr} is unbounded below and consider a sequence $(u_k,y_k)$ with $\nabla P(\zb)u_k+y_k\in T_D(P(\zb))$ and $\skalp{\nabla f(\zb),u_k}+\frac 12\norm{y_k}^2\to -\infty$.
Since the mapping $u\rightrightarrows \nabla P(\zb)u-T_D(P(\zb))$ is assumed to be metrically subregular and its graph is a closed cone, by Lemma \ref{LemConicalMult}
 we can find another sequence $\tilde u_k$ with $\nabla P(\zb)\tilde u_k \in T_D(P(\zb))$ and
\[\norm{\tilde u_k-u_k}\leq\kappa\dist{\nabla P(\zb)u_k,T_D(P(\zb))}\leq \kappa\norm{y_k}.\]
 Because $\zb$ is B-stationary for the program \eqref{EqGenOptProbl} we have $\skalp{\nabla f(\zb),\tilde u_k}\geq 0$, implying
\[\skalp{\nabla f(\zb),u_k}+\frac 12\norm{y_k}^2\geq \skalp{\nabla f(\zb),u_k-\tilde u_k} +\frac 12\norm{y_k}^2\geq -\kappa\norm{\nabla f(\zb)}\norm{y_k}+\frac 12 \norm{y_k}^2\to -\infty,\]
which is obviously not possible. Hence, \eqref{EqAuxProgr} is bounded below. Finally, the last statement about S-stationarity of B-stationary solutions follows immediately from Theorem \ref{ThSuffS_Stat} applied to \eqref{EqAuxProgr}.
\end{proof}
\begin{lemma}\label{LemSolQP}Consider the program
\begin{equation}
\label{EqQP}  \min_{z\in\R^d}q(z):=\frac 12 z^T  Bz+b^T  z\quad\mbox{subject to}\quad Az\in C,
\end{equation}
where $B$ denotes a positive semidefinite $d\times d$-matrix,  $b\in\R^d$, $A$ is an $s\times d$ matrix and $C\subset \R^s$ is a polyhedral set.
Then exactly one of the following alternatives can occur:
\begin{enumerate}
  \item The program \eqref{EqQP} is infeasible
  \item The program \eqref{EqQP} is unbounded below, i.e. there is a sequence $z_k$ satisfying $Az_k\in C$ and $\lim_{k\to\infty} q(z_k)=-\infty$.
  \item There exists a global solution $\bar z$.
\end{enumerate}
\end{lemma}\begin{proof}
  It suffices to show that the program \eqref{EqQP} has a global solution if it is feasible and bounded below. Let $C$ be the union of the convex polyhedral sets $C_1,\ldots C_p$ and consider for each $i$ the convex quadratic program
  \[\min_z q(z)\quad\mbox{subject to}\quad Az\in C_i.\]
  If this program is feasible, then it must possess a global solution $\zb_i$, since otherwise by \cite[Lemma 4]{BeGfr16d} there would exist a direction $w$ satisfying $Aw\in 0^+C_i$ (the recession cone of $C_i$), $Bw=0$ and $b^T  w<0$ contradicting the boundedness of \eqref{EqQP}. Then the one of the $\zb_i$ who has the samllest objective function value is a global solution of \eqref{EqQP}.
\end{proof}

\begin{proof}[Proof of Proposition \ref{PropStrongFirstOrder}]
  Assuming that the first condition \eqref{EqKKT} of Proposition \ref{PropStrongFirstOrder} is not fulfilled we will  show that the second condition must be fulfilled. If the first condition is not fulfilled, then problem \eqref{EqAuxProgr} cannot have a global solution, because every global solution $(\bar u,\yb)$ would be B-stationary and therefore also fulfilling the S-stationary conditions \eqref{EqSStatAuxProbl} and consequently also the first condition \eqref{EqKKT} of Proposition \ref{PropStrongFirstOrder}. On the other hand, the program \eqref{EqAuxProgr} is bounded below and hence we can find a sequence $(u_k,y_k)$ satisfying
  $\nabla P(\zb)u_k+y_k\in T_D(P(\zb))$ $\forall k$ and
  \begin{equation}\label{EqMinSequ}\lim_{k\to\infty} \skalp{\nabla f(\zb),u_k}+\frac 12 \norm{y_k}^2=\gamma:=\inf\{\skalp{\nabla f(\zb),u}+\frac 12 \norm{y}^2\mv \nabla P(\zb)u+y\in T_D(P(\zb))\}.\end{equation}
 It follows that $\gamma<0$ and without loss of generality we can assume that $\skalp{\nabla f(\zb),u_k}<0$ for all $k$ implying $y_k\not=0$ by B-stationarity of $\zb$.
 Next we can assume without loss of generality that $u_k$ is the element $u$ with minimal norm fulfilling $\skalp{\nabla f(\zb),u}=\skalp{\nabla f(\zb),u_k}, \nabla P(\zb)u+y_k\in T_D(P(\zb))$.
 The  sequence $u_k$ must be unbounded because otherwise the sequence $y_k$ must be bounded as well and thus $(u_k,y_k)$ possesses some limit point $(\bar u,\yb)$  which would be a global solution  of \eqref{EqAuxProgr}.  Thus by passing to a subsequence we can assume that $\lim_k\norm{u_k}=\infty$ and that $u_k/\norm{u_k}$ converges to some $\bar u$.  From
 \[0=\limsup_{k\to\infty}\frac{\gamma}{\norm{u_k}^2}= \limsup_{k\to\infty}\Big(\frac{\skalp{\nabla f(\zb),u_k}}{\norm{u_k}^2}+\frac{\norm{y_k}^2}{2\norm{u_k}^2}\Big)=\limsup_{k\to\infty}\frac{\norm{y_k}^2}{2\norm{u_k}^2}\]
 we conclude $\norm{y_k}/\norm{u_k}\to 0$. Hence
\begin{eqnarray*}&&\skalp{\nabla f(\zb),\bar u}=\lim_{k\to\infty}\frac{\skalp{\nabla f(\zb),u_k}}{\norm{u_k}}\leq0,\\
&&\nabla P(\zb)\bar u=\lim_{k\to \infty}\frac 1{\norm{u_k}}\big(\nabla P(\zb)u_k+y_k)\in T_D(P(\zb)),
\end{eqnarray*}
implying $\bar u\in\TlinO(\zb)$.
Since $\zb$ is B-stationary for \eqref{EqGenOptProbl} it follows from GGCQ that $\skalp{\nabla f(\zb),\bar u}=0$ and that $\bar u$ is a global solution of the program
\[\min_u \skalp{\nabla f(\zb),u}\quad \mbox{subject to}\quad u\in \TlinO(\zb).\]
Hence the B-stationarity condition \eqref{EqKKTNormal} follows. Next we show \eqref{EqNotZero} by contraposition. Assuming that $\nabla P(\zb)\ub\in\Lsp(T_D(P(\zb)))$, we have
$\skalp{\nabla f(\zb),u_k-\norm{u_k}\ub}=\skalp{\nabla f(\zb),u_k}$ and $\nabla P(\zb)(u_k-\norm{u_k}\ub)+y_k\in T_D(P(\zb))$. Since $\norm{(u_k-\norm{u_k}\ub)}=\norm{u_k}\norm{\frac{u_k}{\norm{u_k}}-\ub}<\norm{u_k}$ for  $k$ sufficiently large, we get a contradiction to our choice of $u_k$ and therefore $\nabla P(\zb)\ub\not\in\Lsp(T_D(P(\zb)))$.

There remains to show that $T_D(P(\zb))$ is not locally polyhedral near $\nabla P(\zb)\bar u$. Assuming on the contrary that  $T_D(P(\zb))$ is locally polyhedral near $\nabla P(\zb)\bar u$, we can find a polyhedral set $C$ and a neighborhood $W$ of $\nabla P(\zb)\bar u$ such that $T_D(P(\zb))\cap W=C\cap W$. We can choose the neighborhood $W$ as a convex polyhedral set, e.g. as a sufficiently small ball around $\nabla P(\zb)\bar u$ with respect to the maximum norm. Hence we can assume that $C\cap W$ is polyhedral and is the union of the convex polyhedral sets $C_1,\ldots,C_q$ having the representations $C_i=\{w\mv \skalp{a_{ij},w}\leq\alpha_{ij}, j=1,\ldots,p_i\}$. Consider the set
\[\bigcup_{\beta\geq 1}\beta C_i=\pi(\{(w,\beta)\mv \skalp{a_{ij},w}-\beta\alpha_{ij}\leq 0, j=1,\ldots,p_i,\ \beta\geq 1\}),\]
where $\pi(w,\beta):=w$. By \cite[Theorem 19.3]{Ro70} this set is a convex polyhedral set, implying that the set
\[\bigcup_{\beta\geq 1}\beta(T_D(P(\zb))\cap W)=\bigcup_{\beta\geq 1}\beta(C\cap W)=\bigcup_{i=1}^p\bigcup_{\beta\geq 1} \beta C_i\]
is polyhedral. Consider the optimization problem
\begin{equation}\label{EqAuxOptProbl1}\min_{u,y}\skalp{\nabla f(\zb),u}+\frac 12 \norm{y}^2\quad\mbox{subject to}\quad\nabla P(\zb)u+y\in \bigcup_{\beta\geq 1}\beta(T_D(P(\zb))\cap W).\end{equation}
Since $\bigcup_{\beta\geq 1}\beta(T_D(P(\zb))\cap W)\subset \bigcup_{\beta\geq 1}\beta T _D(P(\zb)))=T_D(P(\zb))$, we conclude from Lemma \ref{LemAuxProbl} that the problem \eqref{EqAuxOptProbl1} is bounded below and thus by Lemma \ref{LemSolQP}  it possesses a global solution $(\tilde u,\tilde y)$. By the construction of $\bar u$ we have $(\nabla P(\zb)u_k+y_k)/\norm{u_k}\in C\cap W$ for all $k$ sufficiently large  and thus $(\nabla P(\zb)u_k+y_k)\in \bigcup_{\beta\geq 1}\beta(T_D(P(\zb))\cap W)$. This shows $\skalp{\nabla f(\zb),\tilde u}+\frac 12 \norm{\tilde y}^2\leq \skalp{\nabla f(\zb),u_k}+\frac 12 \norm{y_k}^2$ and from \eqref{EqMinSequ} we obtain that $(\tilde u,\tilde y)$ is a global solution of \eqref{EqAuxProgr}, a contradiction. Therefore, $T_D(P(\zb))$ is not locally polyhedral near $\nabla P(\zb)\bar u$ and this completes the proof.
\end{proof}
For the sake of completeness we state also the following extension of Proposition \ref{PropStrongFirstOrder}, which exploits some additional features in case of problems of the form \eqref{EqGenOptProbl1}. Rewriting this problem in the form \eqref{EqGenOptProbl}, the set $D$ is the graph of $Q$ and then the tangent cone to $D$ is the graph of another multifunction, the so-called graphical derivative.
\begin{proposition}\label{PropStrongFirstOrderGraph}
  In addition to the  assumptions of Theorem \ref{ThStrongFirstOrder} assume that $T_D(P(\zb))$ is the graph of a set-valued mapping $M=M_c+M_p$, where $M_c,M_p:\R^r\rightrightarrows \R^{s-r}$ are set-valued mappings whose graphs are closed cones, $M_p$ is polyhedral and there is some real $C$ such that
  \begin{equation}\label{EqBoundM_C} \norm{t}\leq C \norm{w}\ \forall (w,t)\in\Gr M_c.\end{equation}
  Then either there is $w\in T_D(P(\zb))$  and a multiplier  $w^\ast  \in\widehat N_{T_D(P(\zb))}(w)$ fulfilling \eqref{EqKKT} or there is some $\bar u\in \TlinO(\bar z)$ fulfilling \eqref{EqNotZero},\eqref{EqKKTCritDir} and \eqref{EqKKTNormal} such that $T_D(P(\zb))$ is not locally polyhedral near $\nabla P(\zb)\bar u$ and there is some $\bar w\not=0$ with
  \begin{equation}
    \label{EqKKT_W_NotZero} \nabla P(\zb)\bar u\in\{\bar w\}\times M(\bar w).
  \end{equation}
\end{proposition}
\begin{proof}
  We only have to show \eqref{EqKKT_W_NotZero} and we can proceed quite similar as in the proof of Proposition \ref{PropStrongFirstOrder}. Assuming that we cannot fulfill \eqref{EqKKT}, let $(u_k,y_k)$ denote a sequence satisfying
  $\nabla P(\zb)u_k+y_k\in T_D(P(\zb))$ and \eqref{EqMinSequ}. Let $w_k$ and $t_k\in M_c(w_k)$ be given by $\nabla P(\zb)u_k+y_k\in \{w_k\}\times (t_k+ M_p(w_k))$ and consider for each $k$ the problem
  \begin{equation}\label{EqAuxProgr_k}\min \skalp{\nabla f(\zb),u}+\frac 12\norm{y}^2\ \mbox{subject to}\ \nabla P(\zb)u+y\in (w_k,t_k+M_p(w_k)).\end{equation}
  Since $M_p(w_k)$ is a polyhedral set, by Lemma \ref{LemSolQP} this problem has a global solution and we now claim that  there is also a global solution $(\tilde u_k,\tilde y_k)$ fulfilling
  \[ \norm{(\tilde u_k,\tilde y_k)}\leq \gamma_1+\gamma_2(\norm{t_k}+\norm{w_k}),\]
  where $\gamma_1,\gamma_2$ do not depend on $k$. Indeed, let $\Gr M_p$ be the union of the convex polyhedral sets $C_i$, $i=1,\ldots,p$ with representation
  \[C_i=\{(w,t_p)\mv \skalp{a_{ij},w}+\skalp{b_{ij},t_p}\leq \alpha_{ij},\ j=1,\ldots,q_i\}\]
  and consider for each $i$ and each index set $J\subset \{1,\ldots,q_i\}$ the  set $S(i,J,w_k,t_c)$ consisting of all $(u,y,t_p,\mu_1,\mu_2,\lambda)\in\R^d\times\R^s\times\R^{s-r} \times\R^r\times\R^{s-r}\times \R^{q_i}$ satisfying the system of linear equalities and linear inequalities
  \begin{eqnarray}\label{EqSubKKT1}&&\nabla P(\zb)^\ast  \left(\begin{array}{c}\mu_1\\\mu_2   \end{array}\right)=-\nabla f(\zb),\ y+\left(\begin{array}{c}\mu_1\\\mu_2   \end{array}\right)=0\\
    \label{EqSubKKT2}&&-\mu_2+\sum_{j\in J}\lambda_i b_{ij}=0, \lambda_j\geq 0,\ j\in J, \lambda_j=0,\ j\in\{1,\ldots,q_i\}\setminus J\\
    \label{EqSubKKT3}&&\nabla P(\zb)u+y -(0,t_p)=(w_k,t_c)\\
    \label{EqSubKKT4}&&\skalp{b_{ij},t_p}\begin{cases}
      =\alpha_{ij}-\skalp{a_{ij},w_k}&\mbox{if $j\in J$},\\
      \leq\alpha_{ij}-\skalp{a_{ij},w_k}&\mbox{if $j\not\in J$}.
    \end{cases}
  \end{eqnarray}
  By Hoffman's error bound there is some constant $\gamma^{i,J}$ such that
  \[\dist{0, S(i,J,w_k,t_c)}\leq \gamma^{i,J}(\norm{\nabla f(\xb)}+\norm{w_k}+\norm{t_c}+\sum_{j=1}^{q_i}\vert \alpha_{ij}-\skalp{a_{ij},w_k}\vert\]
  whenever $S(i,J,w_k,t_c)\not=\emptyset$. Note that for every $(u,y,t_p,\mu_1,\mu_2,\lambda)\in S(i,J,w_k,t_c)$ the triple $(u,y, t_p)$ is a global solution of the convex quadratic program
  \begin{equation}\label{EqSubQP}\min \skalp{\nabla f(\zb),u}+\frac 12\norm{y}^2\ \mbox{subject to}\ \nabla P(\zb)u+y-(0,t_p)= (w_k,t_c), (w_k,t_p)\in C_i\end{equation}
  because the equations \eqref{EqSubKKT1}-\eqref{EqSubKKT4} constitute the Karush-Kuhn-Tucker conditions for this problem. Conversely, for every solution of $(u,y, t_p)$ of this program there must exist multipliers $(\mu_1,\mu_2,\lambda)$ such that $(u,y,t_p,\mu_1,\mu_2,\lambda)$ fulfills the Karush-Kuhn-Tucker conditions and thus $(u,y,t_p,\mu_1,\mu_2,\lambda)\in S(i,J,w_k,t_c)$ with $J:=\{j\mv \lambda_j>0\}$.

  Now let $(u,y)$ denote a global solution of \eqref{EqAuxProgr_k} and let $t_p\in M_p(w_k)$ be given by $\nabla P(\zb)u+y-(0,t_p)= (w_k,t_c)$. Consider $i$ such that $(w_k,t_p)\in C_i$. Then
  the triple $(u,y,t_p)$ is a global solution of \eqref{EqSubQP} and we can find some index set $J$ such that $S(i,J,w_k,t_c)\not=\emptyset$. Obviously this set is closed and thus  we can find $(\tilde u,\tilde y,\tilde t_p,\tilde \mu_1,\tilde \mu_2)\in S(i,J,w_k,t_c)$ such that $\norm{(\tilde u,\tilde y,\tilde t_p,\tilde \mu_1,\tilde \mu_2)}=\dist{0,S(i,J,w_k,t_c)}$, implying
 \begin{eqnarray*}\norm{(\tilde u,\tilde y)}&\leq& \norm{(\tilde u,\tilde y,\tilde t_p,\tilde \mu_1,\tilde \mu_2)}\leq \gamma^{i,J}(\norm{\nabla f(\xb)}+\norm{w_k}+\norm{t_c}+\sum_{j=1}^{q_i}\vert \alpha_{ij}-\skalp{a_{ij},w_k}\vert)\\
 &\leq&
  \gamma^{i,J}(\norm{\nabla f(\xb)}+\sum_{j=1}^{q_i}\vert \alpha_{ij}\vert)+\gamma^{i,J}(\norm{t_c}+(1+\sum_{j=1}^{q_i}\norm{a_{ij}})\norm{w_k}.
  \end{eqnarray*}
 Since both $(\tilde u,\tilde y,\tilde t_p)$ and $(u,y,t_p)$ constitute  global solutions of \eqref{EqSubQP} and $(u,y)$ is a global solution of \eqref{EqAuxProgr_k}, $(\tilde u,\tilde y)$ is a global solution of \eqref{EqAuxProgr_k} and our claim follows with
  $(\tilde u_k,\tilde y_k)=(\tilde u,\tilde y)$ and
  \[\gamma_1=\max_{i,J}\gamma^{i,J}(\norm{\nabla f(\xb)}+\sum_{j=1}^{q_i}\vert \alpha_{ij}\vert),\ \gamma_2=\max_{i,J}\gamma^{i,J}(1+\sum_{j=1}^{q_i}\norm{a_{ij}}).\]
  Together with \eqref{EqBoundM_C} we obtain
  \begin{equation}\label{EqBndW_kU_k}\norm{(\tilde u_k,\tilde y_k)}\leq \gamma_1+\gamma_2(1+C)\norm{w_k}.\end{equation}
  Since $(u_k,y_k)$ is feasible for the problem \eqref{EqAuxProgr_k}, we have $\skalp{\nabla f(\zb),\tilde u_k}+\frac 12 \norm{\tilde y_k}^2\leq \skalp{\nabla f(\zb), u_k}+\frac 12 \norm{ y_k}^2$ and thus  $(\tilde u_k,\tilde y_k)$ is another sequence fulfilling \eqref{EqMinSequ}. We can proceed as in the proof of Proposition \ref{PropStrongFirstOrder} to show that, after passing to a subsequence, the sequence $\tilde u_k/\norm{\tilde u_k}$ converges to some $\bar u\in\TlinO(\zb)$ fulfilling \eqref{EqNotZero},\eqref{EqKKTCritDir} and \eqref{EqKKTNormal} and $T_D(P(\zb))$ is not locally polyhedral near $\nabla P(\zb)\bar u$. Because $\nabla P(\zb)\bar u=\lim_{k\to\infty}(\nabla P(\zb)\tilde u_k+\tilde y_k)/\norm{\tilde u_k}$ and
  \[(\nabla P(\zb)\tilde u_k+\tilde y_k)/\norm{\tilde u_k}\in\frac 1{\norm{\tilde u_k}}\Big(\{w_k\}\times M(w_k)\Big)=\{\frac{w_k}{\norm{\tilde u_k}}\}\times M\Big(\frac{w_k}{\norm{\tilde u_k}}\Big)\]
  we conclude that $\frac{w_k}{\norm{\tilde u_k}}$ converges to some $\bar w$ such that $\nabla P(\zb)\bar u\in\{\bar w\}\times M(\bar w)$. From \eqref{EqBndW_kU_k} we obtain $1\leq \gamma_2(1+C)\norm{\bar w}$ implying $\norm{\bar w}>0$. This completes the proof.
\end{proof}

\section{Application to MPEC\label{SecMPEC}}
In this section we want to demonstrate that the linearized M-stationarity conditions can be applied to the MPEC \eqref{EqMPEC'} when it is impossible to compute the limiting normal cone effectively. Recall that this program is given by
\begin{align*}
\mbox{(MPEC')}\qquad \min_{x,y}\ & F(x,y)\\
\nonumber  \mbox{s.t. }&\hat P(x,y):=\left(\begin{array}{c}(y,-\phi(x,y))\\ G(x,y)\end{array}\right)\in\Gr\widehat N_\Gamma\times\R^p_-=:\hat D,
\end{align*}
where $F:\R^n\times\R^m\to \R$, $\phi:\R^n\times\R^m\to \R^m$ and $G:\R^n\times\R^m\to\R^p$ are continuously differentiable and $\Gamma:=\{y\mv g(y)\leq 0\}$ is given by a $C^2$-mapping $g:\R^m\to\R^q$.

For the rest of the section let $(\xb,\yb)$ denote a B-stationary solution for the program (MPEC') such that  the following assumption is fulfilled:
\begin{assumption}\label{Ass1}
\begin{enumerate}\item MSCQ holds for the lower level system $g(y)\in\R^q_-$ at $\yb$.
\item GGCQ holds at $(\xb,\yb)$ and the mapping
\begin{eqnarray*}(u,v)&\rightrightarrows& \nabla \hat P(\xb,\yb)(u,v)-T_{\hat D}(\hat P(\xb,\yb))\end{eqnarray*}
is metrically subregular at $((0,0),0)$.
\end{enumerate}
\end{assumption}
Note that by Remark \ref{RemSubReg} the second part of Assumption \ref{Ass1} is fulfilled if MSCQ holds for the system $\hat P(x,y)\in \hat D$ at $(\xb,\yb)$. A point-based sufficient condition for the validity of MSCQ for this system is given by \cite[Theorem 5]{GfrYe17a}.

We need some more notation. We set $\yba:=-\phi(\xb,\yb)$ and denote by
\[\KbG:=\K_\Gamma(\yb,\yba)\]
the  critical cone for $\Gamma$ at $(\yb,\yba)$. Further we define the {\em multiplier set}
\[\Lb:=\{\lambda\in N_{\R^q_-}(g(\yb))\mv \nabla g(\yb)^\ast  \lambda=\yba\}\]
and for every $v\in \KbG$ the {\em directional multiplier set}
\[\Lbv:=\argmax\{v^T  \nabla^2(\lambda^T  g)(\yb)v\mv \lambda\in \Lb\}.\]
By \cite[Proposition 4.3(iii)]{GfrMo15a} we have $\Lbv\not=\emptyset$ $\forall v\in\KbG$ thanks to Assumption \ref{Ass1}(1).

By \cite[Proposition 1]{GfrYe17a} we have
\[T_{\hat D}(\hat P(\xb,\yb))=T_{\Gr \widehat N_\Gamma}(\yb,\yba)\times T_{\R^p_-}(G(\xb,\yb)).\]
In order to compute the tangent cone $T_{\Gr \widehat N_\Gamma}(\yb,\yba)$ we use the following theorem:
\begin{theorem}[{cf. \cite[Theorem 4]{GfrYe17a}}]\label{ThTanConeGrNormalCone}Assume that MSCQ holds  at $\yb$ for
the system $g(y)\in\R^q_-$. Then  there is a real $\kappa>0$ such that
 the tangent cone to the graph of $\widehat N_\Gamma$ at $(\yb,\yba)$ can be calculated by
\begin{eqnarray}\label{EqTanConeGrNormalCone}
\lefteqn{T_{\Gr \widehat N_\Gamma}(\yb,\yba)}\\
\nonumber
&=&\big\{(v,v^\ast  )\in\R^{2m}\mv\exists\,\lambda\in\Lbv\;\mbox{ with }\;
v^\ast  \in\nabla^2(\lambda^T  g)(\yb)v+N_{\KbG}(v)\big\}\\
\nonumber&=&\big\{(v,v^\ast  )\in\R^{2m}\mv\exists\,\lambda\in\Lbv\cap \kappa\norm{\yba} \B_{\R^q}\;\mbox{ with }\;
v^\ast  \in\nabla^2(\lambda^T  g)(\yb)v+N_{\KbG}(v)\big\}.
\end{eqnarray}
\end{theorem}

We see that the tangent cone $T_{\hat D}(\hat P(\xb,\yb))$ is the graph of the multifunction $M(v)=M_c(v)+M_p(v)$, where
\[M_p(v):=N_{\KbG}(v)\times T_{\R^p_-}(G(\xb,\yb))\]
is a polyhedral multifunction and
\[ M_c(v):=\{\nabla^2(\lambda^T  g)(\yb)v\mv \lambda\in\Lbv\cap \kappa\norm{\yba} \B_{\R^q}\}\times \{0\}.\]
fulfills \eqref{EqBoundM_C}.

\begin{proposition}\label{PropPolyhedr}Let a critical direction $\bar v\in \KbG$ be given. If there is an open neighborhood $V$ of $\bar v$ and a set $\tilde\Lambda\subset \Lb$ such that
\begin{equation}\label{EqConstLambda}\Lb(v)=\tilde\Lambda\ \forall v\in (\KbG\setminus\{\bar v\})\cap V\end{equation}
then
\begin{equation}\label{EqTanConePoly}T_{\Gr \widehat N_\Gamma}(\yb,\yba)\cap(V\times\R^m)
=\{\big(v,\nabla ^2(\tilde\lambda^T  g)(\yb)v+z^\ast  \big)\mv z^\ast  \in N_{\KbG}(v)\}\cap (V\times\R^m),\end{equation}
where $\tilde\lambda\in\tilde\Lambda$ is an arbitrarily fixed multiplier. In particular, $T_{\Gr \widehat N_\Gamma}(\yb,\yba)$ is locally polyhedral near $(\bar v,\vba)$ for every $\vba$ satisfying $(\bar v,\vba)\in T_{\Gr \widehat N_\Gamma}(\yb,\yba)$ and
\begin{equation}\label{EqRegNormalPoly}\widehat N_{T_{\Gr \widehat N_\Gamma}(\yb,\yba)}(\bar v,\vba)=\big\{(w^\ast  ,w)\mv (w^\ast  +\nabla^2(\tilde\lambda^T  g)(\yb)w, w)\in \big(\K_\KbG(\bar v,\bar z^\ast  )\big)^\ast\times \K_\KbG(\bar v,\bar z^\ast  )\big\},\end{equation}
where $\bar z^\ast  :=\vba-\nabla^2(\tilde\lambda^T  g)(\yb)\bar v$.
\end{proposition}
\begin{proof}Let $\tilde \lambda\in\tilde\Lambda$ be arbitrarily fixed. We claim that for every $v\in(\KbG\setminus\{\bar v\})\cap V$ we have
\begin{equation}\label{EqAuxClaim1}\big\{\nabla^2(\lambda^T  g)(y)v\mv  \lambda\in\Lb(v)\big\}+ N_{\KbG}(v)=\nabla^2(\tilde\lambda^T  g)(y)v+ N_{\KbG}(v)\end{equation}
Indeed, consider $v^\ast  =\nabla^2(\lambda^T  g)(y)v+z^\ast  $ with $\lambda\in\Lb(v)$ and $z^\ast  \in N_{\KbG}(v)$. Since $\KbG$ is a convex polyhedral set, for every $w\in T_{\KbG}(v)$ we have
$v+\alpha w\in (\KbG\setminus\{\bar v\})\cap V$ for all $\alpha\geq 0$ sufficiently small and therefore $(v+\alpha w)^T  \nabla^2(\lambda^T  g)(\yb)(v+\alpha w)=(v+\alpha w)^T \nabla^2(\tilde \lambda^T  g)(\yb)(v+\alpha w)$. Because we also have $v^T  \nabla^2(\lambda^T  g)(\yb)v =v^T  \nabla^2(\tilde \lambda^T  g)(\yb)v$ we conclude $v^T  \nabla^2\big((\lambda-\tilde\lambda)^T  g\big)(\yb)w=0$ $\forall w\in T_{\KbG}(v)$ and consequently $\nabla^2\big((\lambda-\tilde\lambda)^T  g\big)(\yb)v\in \Lsp(N_{\KbG}(v))$. Thus
\begin{eqnarray*}v^\ast  &=&\nabla^2(\tilde \lambda^T  g)(y)v+\nabla^2\big((\lambda-\tilde\lambda)^T  g\big)(\yb)v+z^\ast  \\
&\in& \nabla^2(\tilde \lambda^T  g)(y)v+\Lsp(N_{\KbG}(v))+N_{\KbG}(v)= \nabla^2(\tilde \lambda^T  g)(y)v+N_{\KbG}(v)\end{eqnarray*}
and
\[\big\{\nabla^2(\lambda^T  g)(y)v\mv  \lambda\in\Lb(v)\big\}+ N_{\KbG}(v)\subset\nabla^2(\tilde\lambda^T  g)(y)v+ N_{\KbG}(v)\]
follows. Since the reverse inclusion obviously holds, our claim \eqref{EqAuxClaim1} is verified. We next show that \eqref{EqAuxClaim1} holds for $v=\bar v$ as well. If $\bar v=0$ then \eqref{EqAuxClaim1} obviously holds for $v=\bar v$. On the other hand, if $\bar v\not=0$, we can find some $\alpha\not=1$ sufficiently close to $1$ such that  $\alpha\bar v\in
(\KbG\setminus\{\bar v\})\cap V$, implying
\begin{eqnarray*}\lefteqn{\alpha\Big(\big\{\nabla^2(\lambda^T  g)(y)\bar v\mv  \lambda\in\Lb(\bar v)\big\}+ N_{\KbG}(\bar v)\Big)=
\big\{\nabla^2(\lambda^T  g)(y)\alpha\bar v\mv  \lambda\in\Lb(\alpha \bar v)\big\}+ N_{\KbG}(\alpha\bar v)}\\
&=& \nabla^2(\tilde\lambda^T  g)(y)\alpha \bar v+ N_{\KbG}(\alpha \bar v)=\alpha\Big(\nabla^2(\tilde\lambda^T  g)(y)\bar v+ N_{\KbG}(\bar v)\Big),\hspace{3.5cm}\end{eqnarray*}
where we have used the relations $\Lb(\alpha \bar v)=\Lb(\bar v)$ and $N_{\KbG}(\bar v)=N_{\KbG}(\alpha\bar v)=\alpha N_{\KbG}(\bar v)$. Thus \eqref{EqAuxClaim1} holds for all $v\in \KbG\cap V$ and the representation \eqref{EqTanConePoly} follows from \eqref{EqTanConeGrNormalCone}. Since the graph of the normal cone mapping to a convex polyhedral set is a polyhedral set \cite{Rob79}, $\Gr N_\KbG$ is the union of polyhedral convex sets $C_1,\ldots, C_l\subset\R^m\times\R^m$. By taking into account \cite[Theorem 19.3]{Ro70} we obtain that $\{\big(v,\nabla ^2(\tilde\lambda^T  g)(\yb)v+z^\ast  \big)\mv z^\ast  \in N_{\KbG}(v)\}$ is the union of the polyhedral convex sets $\{\big(v,\nabla ^2(\tilde\lambda^T  g)(\yb)v+z^\ast  \big)\mv (v,z^\ast  )\in C_i\}$, $i=1,\ldots,l$.  Now it follows from \eqref{EqTanConePoly} that  $T_{\Gr \widehat N_\Gamma}(\yb,\yba)$ is locally polyhedral near $(\bar v,\vba)$ for every $\vba$ satisfying $(\bar v,\vba)\in T_{\Gr \widehat N_\Gamma}(\yb,\yba)$.

  By virtue of \eqref{EqTanConePoly}, for every pair $(v,v^\ast  )\in T_{\Gr \widehat N_\Gamma}(\yb,\yba)$ close to $(\bar v,\vba)$ there is a unique element $z^\ast  \in N_{\KbG}(v)$ with
  $v^\ast  =\nabla^2(\tilde\lambda^T  g)(\yb)v+z^\ast  $. Thus
  \begin{eqnarray*}\lefteqn{(w^\ast  ,w)\in \widehat N_{T_{\Gr \widehat N_\Gamma}(\yb,\yba)}(\bar v,v^\ast  )\Longleftrightarrow \limsup_{(v,v^\ast  )\longsetto{{T_{\Gr \widehat N_\Gamma}(\yb,\yba)}}(\bar v,\vba)}\frac{\skalp{w^\ast  ,v-\bar v}+\skalp{w,v^\ast  -\vba}}{\norm{(v,v^\ast  )-(\bar v,\vba)}}\leq 0}\\
  &\Longleftrightarrow& \limsup_{(v,z^\ast  )\longsetto{{\Gr N_\KbG}}(\vb,\bar z^\ast  )}\frac{\skalp{w^\ast  ,v-\vb}+\skalp{w,\nabla^2(\tilde\lambda^T  g)(\yb)v+z^\ast  -\nabla^2(\tilde\lambda^T  g)(\yb)\vb-\bar z^\ast  }}{\norm{(v,\nabla^2(\tilde\lambda^T  g)(\yb)v+z^\ast  )-(\vb,\nabla^2(\tilde\lambda^T  g)(\yb)\vb+\bar z^\ast  )}}\leq 0\\
  &\Longleftrightarrow& \limsup_{(v,z^\ast  )\longsetto{{\Gr N_\KbG}}(\vb,\bar z^\ast  )}\frac{\skalp{w^\ast  +\nabla^2(\tilde\lambda^T  g)(\yb)w,v-\vb}+\skalp{w,z^\ast  -\bar z^\ast  }}{\norm{(v,z^\ast  )-(\vb,\bar z^\ast  )}}\leq 0\\
  &\Longleftrightarrow&(w^\ast  +\nabla^2(\tilde\lambda^T  g)(\yb)w, w)\in \widehat N_{\Gr N_\KbG}(\vb,\bar z^\ast  )
  \end{eqnarray*}
  and \eqref{EqRegNormalPoly} follows from the identity $\widehat N_{\Gr N_\KbG}(\vb,\bar z^\ast  )= \big(\K_\KbG(\bar v,\bar z^\ast  )\big)^\ast\times \K_\KbG(\bar v,\bar z^\ast  )$, cf. \cite[Equation (13)]{DoRo96}.
\end{proof}
We are now in the position to state the main result of this section.
\begin{theorem}\label{ThMPECNecCond}
Assume that $(\xb,\yb)$ is B-stationary for the program (MPEC'), assume that Assumption \ref{Ass1} is fulfilled and that  there is a set $\tilde\Lambda\subset \Lb$ such that
\begin{equation}\label{EqConstLambda1}\Lb(v)=\tilde\Lambda\ \forall v\in \KbG\setminus\{0\}.\end{equation}
Then for every $\tilde\lambda\in\tilde\Lambda$ there are $v\in \KbG$, $z^\ast  \in N_\KbG(v)$ and multipliers $w\in \K_\KbG( v,\bar z^\ast  )$, $\mu\in N_{\R^p_-}(G(\xb,\yb))$ such that
\begin{align*}&0=\nabla_x F(\xb,\yb)  -\nabla_x\phi(\xb,\yb)^\ast  w+\nabla_x G(\xb,\yb)^\ast  \mu\\
&0\in \nabla_y F(\xb,\yb)  -\nabla^2(\tilde \lambda^T  g)(\yb) w -\nabla_y\phi(\xb,\yb)^\ast  w+\nabla_y G(\xb,\yb)^\ast  \mu +\big(\K_\KbG(v,\bar z^\ast  )\big)^\ast.
\end{align*}
\end{theorem}
\begin{proof}By \eqref{EqConstLambda1} and Proposition \ref{PropPolyhedr} we obtain that $T_{\Gr \widehat N_\Gamma}(\yb,\yba)$ is locally polyhedral near every $(v,v^\ast)\in T_{\Gr \widehat N_\Gamma}(\yb,\yba)$. Since $T_{\R^p_-}(G(\xb,\yb))$ is  a convex polyhedral set, $T_{\hat D}(\yb,\yba,G(\xb,\yb))$ is polyhedral near every direction $(v,v^\ast,t)\in T_{\hat D}(\yb,\yba,G(\xb,\yb))$. Hence, by Proposition \ref{PropStrongFirstOrder} there exists a direction $(v,v^\ast,t)\in T_{\hat D}(\yb,\yba,G(\xb,\yb))$ and a regular normal
$(w^\ast,w,\mu)\in \widehat N_{T_{\hat D}(\yb,\yba,G(\xb,\yb))}(v,v^\ast,t)= \widehat N_{T_{\Gr \widehat N_\Gamma}(\yb,\yba)}(v,v^\ast)\times  N_{T_{\R^p_-}(G(\xb,\yb))}(t)$
such that
\[0=\nabla F(\xb,\yb)+\nabla\hat P(\xb,\yb)^\ast\left(\begin{array}
  {c} w^\ast\\w\\\mu
\end{array}\right)
=\left(\begin{array}{c}
  \nabla_x F(\xb,\yb)-\nabla_x\phi(\xb,\yb)^\ast  w+\nabla_x G(\xb,\yb)^\ast  \mu\\ \nabla_y F(\xb,\yb)  +w^\ast -\nabla_y\phi(\xb,\yb)^\ast  w+\nabla_y G(\xb,\yb)^\ast  \mu
\end{array}\right)\]
By utilizing \eqref{EqRegNormalPoly} and the well-known identity $N_{\R^p_-}(G(\xb,\yb))=\bigcup_{t\in T_{\R^p_-}(G(\xb,\yb))} N_{T_{\R^p_-}(G(\xb,\yb))}(t)$ the assertion follows.
\end{proof}
Recall that the inequalities $g(y)\leq 0$ satisfy the {\em constant rank constraint qualification} (CRCQ) at a feasible point $\yb$ if for each subset $I\subseteq\{i\in\{1,\ldots, q\}\mv g_i(\yb)=0\}$ there is a neighborhood $V$ of $\bar y$ such that the rank of $\{\nabla g_i(y)\mv i\in I\}$ is a constant value on $V$. It was shown in \cite[Proposition 5.3]{GfrMo15a} that CRCQ at $\yb$ is a sufficient condition for \eqref{EqConstLambda1} to hold. By applying \cite[Proposition 5.3]{GfrMo15a} to the system $\tilde g(y)\leq 0$,
where
\[\tilde g_i(y)=g_i(\yb)+\skalp{\nabla g_i(\yb),y-\yb}+\frac 12 (y-\yb)^\ast  \nabla ^2g_i(\yb)(y-\yb),\ i=1,\ldots,q\]
it follows that it is sufficient to require CRCQ for  the system $\tilde g(y)\leq 0$ in order to guarantee \eqref{EqConstLambda1}.
However, it is easy to find examples where the condition \eqref{EqConstLambda1} is fulfilled but CRCQ neither for the system $g(y)\leq 0$ nor the system $\tilde g(y)\leq 0$ holds.

The following example  demonstrates the benefit of the necessary optimality conditions of Theorem \ref{ThMPECNecCond}
\begin{example}Consider the problem
\[\min_{x\in\R,y\in\R^3} x-2y_3\quad\mbox{subject to}\quad 0\in (y_1,y_2,-x+y_3)+\widehat N_\Gamma(y)\]
with \[\Gamma:=\left\{y\in\R^3\mv \begin{array}{l}g_1(y):=y_3-y_1^3\leq 0\\g_2(y):=y_3-a^3y_2^3\leq 0
\end{array}\right\},\]
where $a>0$ denotes a fixed parameter. Then $\xb=0$, $\yb=(0,0,0)$ is a local solution. Obviously MFCQ is fulfilled at $\yb$ and straightforward calculations yield $\yba=(0,0,0)$, $\KbG=\R\times\R\times\R_-$ and
\[\Lb=\Lb(v)=\{(0,0)\}\ \forall v\in \KbG.\]
Thus condition \eqref{EqConstLambda1} is fulfilled and the first-order optimality condition of Theorem \ref{ThMPECNecCond} must hold.
Indeed, taking $\tilde\lambda=(0,0)$, $v=z^\ast=(0,0,0)$ we have $\K_{\KbG}(v,z^\ast)=\KbG$, $\big(\K_{\KbG}(v,z^\ast)\big)^\ast=\{0\}\times\{0\}\times \R_+$ and with $w=(0,0,-1)$ we obtain
\begin{align*}&\nabla_x F(\xb,\yb)  -\nabla_x\phi(\xb,\yb)^\ast  w =1 - (0\ 0\ -1)\left(\begin{array}{c}0\\0\\-1\end{array}\right)=0,\\
&-(\nabla_y F(\xb,\yb)  -\nabla^2(\tilde \lambda^T  g)(\yb) w -\nabla_y\phi(\xb,\yb)^\ast  w)\\
&=-\left(\left(\begin{array}{c}0\\0\\-2\end{array}\right)-\left(\begin{array}{c}0\\0\\0\end{array}\right)-\left(\begin{array}{ccc}1&0&0\\0&1&0\\0&0&1\end{array}\right)\left(\begin{array}{c}0\\0\\-1\end{array}\right)\right)
=\left(\begin{array}{c}0\\0\\1\end{array}\right)\in\big(\K_\KbG(v,\bar z^\ast  )\big)^\ast
\end{align*}
verifying the first-order optimality conditions of Theorem \ref{ThMPECNecCond}.

In \cite[Example 1]{GfrOut16a} the limiting normal cone $N_{\Gr \Gamma}(\yb,\yba)$ was computed explicitly. It appears that it depends on the parameter $a$ and thus a point-based representation of the limiting normal cone in terms of first-order and second-order derivatives  of $g$ is not possible. This shows the difficulty of verifying the M-stationarity conditions at the solution.
\end{example}

So far we have only considered linearized M-stationarity  conditions for the MPEC \eqref{EqMPEC'} under the assumption \eqref{EqConstLambda1}, which allows the application of Theorem \ref{ThStrongFirstOrder} with $k=0$. In a forthcoming paper we will formulate the linearized M-stationarity  conditions for this problem for the general case. Anticipating the main result of that paper we will show with the help of Proposition \ref{PropStrongFirstOrderGraph} that  Theorem \ref{ThStrongFirstOrder} holds with $k=1$.

\section{\label{SecConcl} Concluding remarks and future research}

In this paper we considered new first-order optimality conditions for general optimization problems which are stronger than the commonly used M-stationarity conditions. The key idea is to apply the M-stationarity conditions not to the original problem but to the linearized problem and to repeat this procedure. As a final result we obtain that the multiplier is not only a limiting normal but also a regular normal to  tangent cone to a series of tangent cones. Because the optimality conditions are based on a repeated linearization process we use the term {\em linearized M-stationarity conditions}.

The applicability of the new optimality conditions are demonstrated at the basis of a special MPEC, where the equilibrium is modeled via a general equation involving the normal cone to a set given by $C^2$-inequalities. Under a certain additional condition we explicitly stated the optimality conditions in terms of the problem data at the reference point. This additional assumption ensures that the linearization process has not to be repeated. We presented an example where the M-stationarity conditions  cannot be stated effectively by the difficulty of computing the limiting normal cone, whereas our results fully apply. We plan to drop this additional assumption in a forthcoming paper to obtain the linearized M-stationarity condition for this MPEC in the general case.

A further goal is the application of the developed theory to other problem classes, e.g. to MPECs involving the normal cone to sets appearing in second-order cone programming and semidefinite programming. In particular in the latter case we expect that the linearization process has to be eventually repeated more than once.

Another direction of future research could be the investigation of the sufficiency of the linearized M-stationarity conditions for B-stationarity. Similar as in \cite{Ye05} one could look for properties of the problem functions which ensure that the reference point is a globally or locally optimal solution. Another approach could be the fulfilment of some linearized M-stationary condition in every nonzero critical direction similar to the concept of {\em extended M-stationarity} used in \cite{Gfr14a}.

{\bf Acknowledgements.} The research was partially supported by the Austrian Science Fund (FWF) under grant P29190-N32.


\begin{thebibliography}{99}
\bibitem{AdHenOut17} {\sc L. Adam, R. Henrion, J. Outrata}, {\em On M-stationarity conditions in MPECs and the associated qualification conditions}, Math. Program. Series B, 168 (2018), pp.~229--259.
%
\bibitem{BeGfr17a}{\sc M. Benko, H. Gfrerer}, {\em On estimating the regular normal cone to constraint systems
and stationary conditions}, Optimization, 66 (2017), pp.~61--92.
%
\bibitem{BeGfr16d}{\sc M. Benko, H. Gfrerer}, {\em New verifiable stationarity concepts for a class of mathematical programs with disjunctive constraints}, Optimization, 67 (2018), pp.~1--23.
%
\bibitem{Cla73}{\sc F.~H. Clarke}, {\em Necessary conditions for Nonsmooth Problems in Optimal Control and
the Calculus of Variations}, Ph.D. dissertation, University of Washington, Seattle, 1973.
%
\bibitem{ChiHi17}{\sc N.H. Chieu, L.V. Hien}, {\em Computation of graphical derivative for a class of
normal cone mappings under a very weak condition}, SIAM J. Optim., 27 (2017), pp.~190-204.

%
\bibitem{DoRo96}{\sc A.~L. Dontchev, R.~T. Rockafellar}, {\em  Characterizations of strong regularity for variational inequalities over polyhedral convex sets},
    SIAM J. Optim., 6 (1996), pp.~1087--1105.
%
\bibitem{DoRo04}
{\sc A.L. Dontchev, R.T. Rockafellar}, {\em  Regularity and conditioning of solution mappings in variational analysis}, Set-Valued Var. Anal.,   12 (2004), pp.~ 79-109.
%
\bibitem{FleKanOut07}{\sc M.~L. Flegel, C. Kanzow, J.~V. Outrata}, {\em Optimality conditions for disjunctive programs
with application to mathematical programs with equilibrium constraints}, Set-Valued Anal., 15 (2007), pp.~139--162.
%
%
\bibitem{Gfr11}{\sc H. Gfrerer}, {\em First order and second order characterizations of metric subregularity and calmness of constraint set mappings},
    SIAM J. Optim., 21 (2011), pp.~1439--1474.
%
%
\bibitem{Gfr14a}{\sc H. Gfrerer}, {\em Optimality conditions for disjunctive programs based on generalized differentiation with application to
mathematical programs with equilibrium constraints}, SIAM J. Optim., 24 (2014), pp.\ 898--931.
%
%
\bibitem{GfrMo15a}{\sc H. Gfrerer, B.S. Mordukhovich}, {\em Complete characterizations of tilt stability in nonlinear programming under weakest qualification conditions},  SIAM J. Optim., 25 (2015), pp.~2081--2119.
%
\bibitem{GfrOut16a}{\sc H. Gfrerer and J. V. Outrata}, {\em On computation of limiting coderivatives of the normal-cone mapping to inequality
systems and their applications}, Optimization,  65 (2016), pp. 671--700.
%
%
\bibitem{GfrOut16b}{\sc H. Gfrerer, J.~V. Outrata}, {\em  On computation of generalized derivatives of the normal-cone mapping and their applications},   Math. Oper. Res. 41 (2016), pp.~1535--1556.
%
\bibitem{GfrYe17a}{\sc H. Gfrerer, J.~J. Ye}, {\em New constraint qualifications for
    mathematical programs with equilibrium constraints via variational analysis}, SIAM J. Optim., 27 (2017), pp.~842--865.
%
\bibitem{HenOut05}{\sc R. Henrion, J.~V. Outrata},  {\em Calmness of constraint systems with applications},
    Math. Program., Ser. B, 104 (2005), pp.~437--464.
%
\bibitem{HenOutSur09}{\sc R. Henrion, J.~V. Outrata, T. Surowiec}, {\em On the coderivative of normal cone mapping to inequality systems}, Nonlinear Anal., 71 (2009), pp.~1213--1226.
%
\bibitem{LuPaRa96}{\sc Z.-Q. Luo, J.-S. Pang, D. Ralph}, {\em Mathematical Programs with Equilibrium Constraints},
Cambridge University Press, Cambridge, UK, 1996.
%
\bibitem{MoOut01}
{\sc B.~S. Mordukhovich, J.~V. Outrata}, {\em On second-order subdifferentials and their applications}, SIAM J. Optim., 12 (2001), 139--169.
%
\bibitem{Out99}{\sc J.~V. Outrata}, {\em  Optimality conditions for a class of mathematical programs with
equilibrium constraints}, Math. Oper. Res., 24 (1999), pp.~627--644.
%
\bibitem{Sch00}{\sc S. Scholtes}, {\em Convergence properties of a regularization scheme for mathematical programs with
complementarity constraints}, SIAM J. Optim., 11 (2001), pp.~918--936.
\bibitem{SchSch00}{\sc H. Scheel, S. Scholtes}, {\em Mathematical programs with complementarity constraints:
Stationarity, optimality, and sensitivity}, Math. Oper. Res., 25 (2000),
pp.~1--22.
%
\bibitem{Ro70}{\sc R.~T. Rockafellar}, {\em Convex analysis}, Princeton, New Jersey, 1970.
\bibitem{Rob79}{\sc S.~M. Robinson}, {\em Generalized equations and their solutions, part I: basic theory}, in Point-to-set Maps and Mathematical Programming, P. Huard, ed., Mathematical Programming Study, 10,  North Holland, Amsterdam, 1979, pp.~128--141.
%
%
\bibitem{RoWe98}{\sc R.T. Rockafellar, R.~J-B. Wets}, {\em Variational analysis}, Springer, Berlin, 1998.
%
\bibitem{Ye99}{\sc J.~J. Ye}, {\em Optimality conditions for optimization problems with complementarity constraints}, SIAM J. Optim., 9 (1999), pp.~374--387.
%
\bibitem{Ye05}{\sc J.~J. Ye}, {\em Necessary and sufficient optimality conditions for mathematical programs with equilibrium
constraints}, J. Math. Anal. Appl. 307 (2005), pp.~350--369.
\end{thebibliography}
\end{document}